\documentclass{mcom-l}

\usepackage{amssymb}

\usepackage{graphicx}

\newtheorem{theorem}{Theorem}[section]
\newtheorem{lemma}[theorem]{Lemma}
\newtheorem{corollary}[theorem]{Corollary}

\theoremstyle{definition}

\theoremstyle{remark}
\newtheorem{remark}[theorem]{Remark}

\numberwithin{equation}{section}

\usepackage{mathrsfs}
\usepackage{enumerate}
\usepackage{subfig}
\usepackage[normalem]{ulem}

\newcommand{\tstep}{\kappa}
\newcommand{\veps}{\varepsilon}
\newcommand{\cqd}{^{\text{\tiny CQ}}D_{t,\tstep}}
\newcommand{\cqw}{\omega^{\text{\tiny CQ}}}

\newcommand{\ub}{\mathbf{u}}
\usepackage{color}

\newcommand{\dt}{D_t}
\newcommand{\Ld}{D_{t,\tstep}^\beta}

\newcommand{\inner}[2]{\left\langle #1,#2 \right\rangle}

\begin{document}

\title[A posteriori error analysis for subdiffusion]%{A posteriori error analysis for the time-discretization of subdiffusion problems}
{A posteriori error analysis for  approximations of time-fractional subdiffusion problems}

\author{Lehel Banjai}
\address{Maxwell Institute for Mathematical Sciences, School of Mathematical  
        \& Computer Sciences, Heriot-Watt University, Edinburgh EH14 4AS, UK}
%\curraddr{}
\email{l.banjai@hw.ac.uk}
%\thanks{}

%    author two information
\author{Charalambos G.\  Makridakis}
\address{Institute for Applied and Computational Mathematics-FORTH, Heraklion-Crete, GR 70013, Greece and Department of Mathematics, University of Sussex, Brighton BN1 9QH, UK}
% \curraddr{}
\email{c.g.makridakis@iacm.forth.gr}
% \thanks{}

%    \subjclass is required.

\date{}

\dedicatory{}

%    Abstract is required.
\begin{abstract}
In this paper we consider a sub-diffusion problem where the fractional time derivative is approximated  either by the L1 scheme or by Convolution Quadrature. 
We propose  new interpretations of the numerical schemes which lead to   a posteriori error estimates.  Our approach is based on  
appropriate   {pointwise representations} of the numerical schemes as perturbed evolution equations and on stability estimates for the evolution equation.  
A posteriori error estimates in $L^2(H)$ and  $L^\infty (H)$ norms of optimal order are derived. Extensive numerical experiments indicate the reliability and the optimality of the estimators for the schemes considered, as well as 
their  efficiency as error indicators driving adaptive mesh selection locating singularities of the problem.  
%For these reasons it is important to be able to use non-uniform time-steps and to locate the regions requiring refinement by the use of reliable a posteriori error estimators. We derive new a posteriori error estimates for \eqref{eq:sub_diff} valid for a general time-discretization scheme. We apply it to both the L1 scheme and low order convolution quadrature based on the backward Euler scheme. Numerical experiments indicate optimality of the estimator for these schemes. 
%\keywords{subdiffusion \and a posteriori error analysis \and L1 \and convolution quadrature}
\end{abstract}

\maketitle

\section{Introduction}
We consider the sub-diffusion problem: Find $u$ such that
\begin{equation}
  \label{eq:sub_diff}
  \begin{aligned}
    \dt^\beta u(t) + Au(t) &= f(t),    & 0 < t < T,\\
    u(0) &= u^0,
  \end{aligned}
\end{equation}
where $\beta \in (0,1)$, $f: \mathbb{R}_{\geq 0} \rightarrow H$ a given inhomogeneity, and $u^0 \in H$ the initial data. Here $A$ is a positive definite, selfadjoint, linear operator on a Hilbert space $H$ with inner product $\inner{\cdot}{\cdot}$ and norm $\|\cdot\|$. {Equations \eqref{eq:sub_diff} are understood as equalities in the Hilbert space $H$.} We also denote $|u|_1 =\inner{Au}{u}^{1/2}$ for  $u \in V = D(A^{1/2})$ and assume that $D(A)$ is dense in $H$. We further assume that there exists a $\sigma > 0$ such that $e^{-\sigma t} \|f(t)\|$ is bounded for $t \geq 0$. The sub-diffusion equation has been used numerous times as a model in natural sciences to describe anomalous diffusion processes, see e.g.,  \cite{Nigmatullin1986,Hatano,Kilbas_book,Metzler,RomanAlemany}. As a simple example of the above setting we can take $H = L^2(\Omega)$, $V = H_0^1(\Omega)$, and $A\cdot = -\Delta \cdot$ the Dirichlet Laplacian. Here $\Omega$ denotes a bounded Lipshitz domain and $H_0^1(\Omega)$ the space of $H^1(\Omega)$ functions with vanishing boundary trace. The setting also applies to the corresponding symmetric conforming and non-conforming Galerkin discretizations.

The fractional derivative $\dt^\beta$ is the Caputo fractional derivative $^C_0\dt^\beta$ of order $\beta \in (0,1)$ given by
\begin{equation}
  \label{eq:frac_diff}
  \dt^\beta g(t) = \frac1{\Gamma(1-\beta)}\int_0^t (t-\tau)^{-\beta} g'(\tau) d\tau.
\end{equation}
Note that
\[
  \dt^\beta g(t) = \dt^{\beta-1} g'(t),
\]
where $\dt^{\beta-1}$ is the Riemann-Liouville integral of order $1-\beta$
\begin{equation}
  \label{eq:frac_int}
  \dt^{\beta-1} h(t) = I_t^{1-\beta} h(t) = \frac1{\Gamma(1-\beta)}\int_0^t (t-\tau)^{-\beta} h(\tau) d\tau.
\end{equation}
A physical derivation of fractional derivatives via discrete random walks is described in \cite{GMMP_randwalk}.

 In the case $f \equiv 0$,  semigroup techniques  are used in \cite[Theorem~4.1]{Baz98} to  show that the unique solution  is analytic as a function of $t$ in an open sector around the positive real axis and bounded there as
 \begin{equation}
\label{eq:baz_bound}
\|A\dt^\beta u\| \leq C|t|^{-\beta}{\|u^0\|}.
 \end{equation}
For non-zero $f$ a representation formula for $u$ is given in  \eqref{eq:representation_formula}; see Lemma~\ref{thm:stab_linf}.

Two popular discretization schemes for fractional derivatives and the sub-diffusion equation \eqref{eq:sub_diff} are convolution quadrature \cite{Lub86_frac,Lub88I} and the L1 scheme \cite{LinXu,SunWu}, both usually using a uniform time-step.  The bound \eqref{eq:baz_bound} indicates that the solution is non-smooth at $t = 0$. Indeed unless a correction term to the standard schemes using uniform time-steps is added \cite{Jin2017,Yan2018} only low order convergence can be expected even in the case of a smooth $f$. Furthermore, if $f$ is not smooth for $t > 0$, similar singularities can also occur elsewhere; see Section~\ref{sec:nonsmooth}.

%Two popular discretization schemes for fractional derivatives and the sub-diffusion equation \eqref{eq:sub_diff} are convolution quadrature \cite{Lub86_frac,Lub88I} and the L1 scheme \cite{LinXu,SunWu}, both usually using a uniform time-step.  The bound \eqref{eq:baz_bound} indicates that the solution is non-smooth at $t = 0$. Indeed unless a correction term to the standard schemes using uniform time-steps is added \cite{Jin2017,Yan2018} only low order convergence can be expected even in the case of a smooth $f$. Furthermore, if $f$ is not smooth for $t > 0$, similar singularities can also occur elsewhere; see Section~\ref{sec:nonsmooth}.

It is important to be able to use non-uniform time-steps and to locate the regions requiring refinement by the use of reliable a posteriori error estimators.  Such estimators are not available for fractional time dependent problems. One of the reasons is the quite involved nature of time dicretisations of fractional derivatives, and consequently the difficulty to connect the numerical schemes to the exact evolution equation. 
In this paper we are able to derive a posteriori error estimates for low order time discrete schemes.  In previous works for time-dependent partial differential equations of parabolic or hyperbolic nature  \emph{reconstruction operators}  are 
introduced to recover 
continuous objects from the approximate solutions, see the review \cite {MR2404052} and e.g., \cite{NSaV, LM:mc2006,AMN:nm2009,Whiler2010,LM2011,GLMV:siam2016,BKM:c2018,MR3584583,MR3803285}. Then the derivation of the estimates is reduced to 
(i) estimate of the reconstruction error  and (ii) the application of PDE stability estimates.  This requires a \emph{pointwise representation} of the numerical scheme as a perturbed evolution equation. Appropriate 
forms of this kind are  not obvious for fractional equations and require new interpretations of the numerical schemes.  
In this paper we address this problem and we derive new a posteriori error estimates for \eqref{eq:sub_diff} valid for a general time-discretization scheme. We apply it to both the L1 scheme and low order convolution quadrature based on the backward Euler scheme. Extensive numerical experiments indicate the reliability and the optimality of the estimators for these schemes as well as 
their  efficiency as error indicators driving adaptive mesh selection locating singularities of the problem. 

%For these reasons it is important to be able to use non-uniform time-steps and to locate the regions requiring refinement by the use of reliable a posteriori error estimators. We derive new a posteriori error estimates for \eqref{eq:sub_diff} valid for a general time-discretization scheme. We apply it to both the L1 scheme and low order convolution quadrature based on the backward Euler scheme. Numerical experiments indicate optimality of the estimator for these schemes. 

The paper is organised as follows. L1 scheme and convolution quadrature are described in Section~\ref{sec:L1} and Section~\ref{sec:cq}. Both schemes are cast in a similar form that allows direct comparison.  The a posteriori estimates will rely on  stability estimates for the evolution problem \eqref{eq:sub_diff} proved in Section~\ref{sec:apriori}, see Theorem \ref{thm:stab_l2} and Theorem \ref{thm:stab_linf}. We provide detailed proofs of both results controlling the $L^2(H)$ and  $L^\infty (H)$ norms of the solution respectively. It is important for the subsequent a posteriori  analysis to note that both results allow the inclusion of a kernel in the forcing term. 
 In Section~\ref{sec:apost} we cast both  L1 scheme and convolution quadrature methods in a unified formulation and we derive corresponding appropriate pointwise forms and error equations. Then we apply  the 
 stability estimates of Section~\ref{sec:apriori} to readily conclude the main   a posteriori error estimates in $L^2(H)$ and  $L^\infty (H)$ norms, Theorem \ref{thm:apost_l2} and Theorem \ref{thm:apost_linf}.
 Numerical   results are presented in Section~\ref{sec:numerics} for all methods and estimators considered.

\section{Numerical method based on piecewise linear time-discretization}\label{sec:L1}

Let $0 = t_0 < t_1 < \dots < t_N = T$ be a partition of $[0,T]$, $I_n: = (t_n,t_{n+1}]$ and $\tstep_n: = t_{n+1}-t_n$. To construct a numerical scheme, given a sequence of values $U_0,\dots,U_N$ we need to define an approximation of the fractional derivative $\dt^\beta$ on this grid. A standard approach is to apply the fractional derivative to a linear interpolant of this data. This is a standard approach to constructing numerical methods including Fredholm and Volterra integal equations \cite{brunner:book,hackbusch:intbook}. In the case of fractional derivatives this approach is often called the L1 scheme \cite{LinXu,SunWu}. We give the details next.

Denote by $\hat U(t)$ the piecewise linear interpolant defined by
\begin{equation}
  \label{eq:lin_int}
  \hat U(t_n) = U_n, \quad n = 0,\dots,N, \qquad \hat U|_{I_n} \in \mathbb{P}_1(I_n),
\end{equation}
where $\mathbb{P}_1$ is the space of linear functions. We also define a projection operator $\Pi_1$ mapping continuous functions to piecewise linear functions by interpolating in $t_n$. Namely
\begin{equation}
  \label{eq:lin_int_cont}
  \Pi_1 u(t_n) = u(t_n), \quad n = 0,\dots,N, \qquad  \Pi_1 u|_{I_n} \in \mathbb{P}_1(I_n).
\end{equation}
Recalling that $\dt^{\beta} = \dt^{\beta-1} \dt$, we now define the discrete fractional derivative by
\begin{equation}
  \label{eq:L1_der}
  \Ld U(t) := \dt^\beta \hat U(t) = \dt^{\beta-1} \hat U'(t).
\end{equation}
Note that  for $t  \in I_n$
\[
  \begin{split}    
    \Ld U(t)    =& \frac1{\Gamma(1-\beta)}\sum_{j = 0}^{n-1}\int_{t_j}^{t_{j+1}}(t-\tau)^{-\beta}\,d\tau \tfrac{U_{j+1}-U_j}{\tstep_j} \\ &+\frac1{\Gamma(1-\beta)}\int_{t_n}^{t}(t-\tau)^{-\beta}\,d\tau \tfrac{U_{n+1}-U_n}{\tstep_n}\\
=& \sum_{j = 0}^{n} \omega_j(t)\tfrac{U_{j+1}-U_j}{\tstep_j},
      \end{split}
    \]
where
\begin{equation}
  \label{eq:omegat}
  \begin{split}    
  \omega_j(t) &= \frac1{\Gamma(1-\beta)}\int_{t_j}^{\min(t,t_{j+1})}(t-\tau)^{-\beta}\,d\tau\\
&= \frac1{\Gamma(2-\beta)} \left((t-t_j)^{1-\beta}-(t-\min(t,t_{j+1}))^{1-\beta}\right)
\end{split}
\end{equation}
for $t \geq t_j$ and
\begin{equation}
  \label{eq:omegat1}
\omega_j(t) = 0 \text{ for } t \leq t_j.
\end{equation}
This implies that $\Ld U(t)$ is continuous as
\[
  \begin{split}    
\Ld U(t_n^+) &= \sum_{j = 0}^{n-1} \omega_j(t_n)\tfrac{U_{j+1}-U_j}{\tstep_j} +\omega_n(t_n)\tfrac{U_{n+1}-U_n}{\tstep_n} \\&= \sum_{j = 0}^{n-1} \omega_j(t_n)\tfrac{U_{j+1}-U_j}{\tstep_j}= \Ld U(t_{n}^-).
  \end{split}
\]

We denote the evaluation of the weights $\omega_j$ at $t_{n+1}$  by $\omega_{n,j}$:
\begin{equation}\label{eq:omegan}
 \begin{split}        
      \omega_{n,j} &:= \omega_j(t_{n+1})\\
&= \frac1{\Gamma(2-\beta)} \left((t_{n+1}-t_j)^{1-\beta}-(t_{n+1}-t_{j+1})^{1-\beta}\right),
      \end{split}
    \end{equation}
%    Similarly for a continuous function $u$, $\partial_{t,\tstep}^{\beta}u(t) := \dt^{\beta-1} \hat \u'(t)$. 
for $n \geq j$. By definition $\omega_{n,j} = 0$ for $n < j$. With all the notation introduced the discrete fractional derivative \eqref{eq:L1_der} evaluated at $t_{n+1}$ is given by
\begin{equation}   \label{eq:L1_der1}
\Ld U(t_{n+1}) = \sum_{j = 0}^n  \omega_{n,j} \frac{U_{j+1}-U_j}{\tstep_j}.
\end{equation}

For later it is also useful to define
\[
  a_j(t) = \frac1{\Gamma(2-\beta)} (t-t_j)^{1-\beta}, \quad t \geq t_j,
\]
and $a_j(t) = 0$ for $t \leq t_j$. Using this definition we have for $t \leq t_{n+1}$
\begin{align}
    \Ld U(t)   =& \sum_{j = 0}^{n} \omega_j(t)\tfrac{U_{j+1}-U_j}{\tstep_j}\nonumber\\
=& \sum_{j = 0}^{n} (a_{j}(t)-a_{j+1}(t))\tfrac{U_{j+1}-U_j}{\tstep_j}\nonumber\\
=& a_0(t)\tfrac{U_{1}-U_0}{\tstep_0}+\sum_{j = 1}^na_j(t)\left(\tfrac{U_{j+1}-U_j}{\tstep_j}-\tfrac{U_{j}-U_{j-1}}{\tstep_{j-1}}\right).\label{eq:upp_repr}
  \end{align}

% We state an approximation result from \cite[Lemma~2.2]{SunWu}.

% \begin{lemma}\label{lem:partial_err}
% Given $g \in C^2[0,T]$ and assuming equal timesteps $\tstep = \tstep_n$, $n = 0,\dots,N-1$, we have that
% \[
% \left|\partial_{t,\tstep}^\beta g(t_n)-\partial_{t}^\beta g(t_n)\right| \leq C_\beta \max_{0\leq \tau \leq t}|g''(\tau)| \tstep^{2-\beta}, \qquad t_n \in (0,T],
% \]
% where
% \[
%   C_\beta = \frac1{12}+\frac1{1-\beta}\left(\frac{2^{2-\beta}}{2-\beta}-1-2^{-\beta}\right).
% \]
% \end{lemma}

We can now write down the fully discrete system: Find $U_{n+1} \in H$, $n = 0,1,\dots,N-1$ such that
    \begin{equation}
      \label{eq:discrete_sys}
      \Ld U(t_{n+1})+AU_{n+1} = f_{n+1},
    \end{equation}
    where $f_{n+1} = f(t_{n+1})$ and $U_0 = u^0$ or some approximation of the initial data. Alternatively, recalling the definition of $\omega_{n,j}$ \eqref{eq:omegan}  we can rewrite the system  in a more familiar form as a finite difference formula
    \begin{equation}
      \label{eq:discrete_sys1}
\sum_{j = 0}^{n} \omega_{n,j} \frac1{\tstep_j}(U_{j+1}-U_j) + AU_{n+1} = f_{n+1}, \qquad n = 0,\dots, N-1.
\end{equation}
This can be seen as a classical collocation scheme for Volterra integral equations \cite{brunner:book,hackbusch:intbook} and is equivalent to the L1 scheme of \cite{SunWu} and \cite{LinXu}. A fast and memory efficient implementation of the solution of such a discretization  is developed in \cite{LoLuSch08}.

We summarize available a-priori results. {Most of these results are for uniform time-step $\tstep_j = \tstep$ for all $j$.} Convergence order of $O(\tstep^{2-\beta})$ is proved in \cite{SunWu} under the assumption that $u \in C^2[0,T]$. In \cite{BangtiLZ:2016} the authors  argue that in general even for smooth  data this smoothness of the solution does not hold and only linear $O(\tstep)$ convergence order is obtained. Namely, they prove that for $f \equiv 0$
\[
  \|u(t_n)-U_n\| \leq C\tstep t_n^{-1} \|u^0\|, \qquad n \geq 1
\] 
 and 
\[
\|u(t_n)-U_n\| \leq C\tstep t_n^{\beta-1} \|Au^0\|, \qquad n \geq 1,
\]
{if $u^0 \in V$}. 
The main reason for this is a singularity at $t = 0$, as  even for $f \equiv 0$, in general $\|\dt^\beta u(t)\| \leq C t^{-\beta}\|u^0\|$ \cite{SakYam}. This suggests that grading towards $t = 0$ would be advantageous. Indeed in \cite{StynesOG} it is proved that the optimal convergence is recovered when using a graded mesh, where the operator $A$ is the differential operator $Au = - \frac{\partial^2}{\partial x^2}u+ c(x) u$ in one spatial dimension. The authors prove that for optimal convergence uniformly for $t > 0$ in the case of a smooth $f$, it is necessary to choose the graded mesh 
\begin{equation}
  \label{eq:stynes_graded}
  t_j =  T (j/N)^k  \text{ with } k \geq \frac{2-\beta}{\beta}, \quad j = 0,\dots,N.
\end{equation}

A modified L1 scheme described in \cite{Yan2018} recovers  the $O(\tstep^{2-\beta})$ convergence order for sufficiently smooth $f$
\[
\|u(t_n) - U_n\| \leq C t_n^{\beta-2} \tstep^{2-\beta},
\]
with $C$ a constant depending on $f$ and $u_0$. The modified scheme reads
    \begin{equation}
      \label{eq:discrete_sys_corr}
      \begin{split}
      \Ld U(t_1)+AU_1 &= f(t_{1})-\frac12(Au_0-f(0)),\\
              \Ld U(t_{n+1})+AU_{n+1} &= f_{n+1}, \qquad n = 1, \dots, N-1.
      \end{split}
    \end{equation}

\section{Convolution quadrature}\label{sec:cq}
Another popular discretization method for fractional derivatives is convolution quadrature 
 \cite{Lub86_frac,Lub88I} with non-uniform time-step schemes investigated in \cite{LopS16}.  The low order convolution quadrature (CQ) based on backward Euler discretization can be given in the following form
 \begin{equation}
   \label{eq:CQ_defn}
[\cqd^\beta V]_{n+1} := \sum_{j = 0}^{n}\cqw_{n,j} \frac{V_{j+1}-V_j}{\tstep_j},   
 \end{equation}
where $\cqw_{n,j}$ are convolution weights for the fractional integral of order $1-\beta$ and is given by
\begin{equation}
  \label{eq:cq_weights}
 \cqw_{n,j}= \tstep_j\frac1{2\pi i} \int_{\sigma-i \infty}^{\sigma+i\infty } z^{\beta-1} \prod_{k = j}^{n} \frac1{1-\tstep_kz} dz, 
% &= \tstep_j\frac1{2\pi i} \int_{\mathcal{C}} z^{\beta-1} \prod_{k = j}^{n} \frac{\tstep_k^{-1}}{\tstep_k^{-1}-z} dz
\end{equation}
for a fixed $\sigma \in  (0, \min_k \tstep_k^{-1})$; due to the analyticity of the integrand, the value of $\cqw_{n,j}$ is independent of $\sigma$. The expression \eqref{eq:CQ_defn} is of the same form as the L1 discrete derivative \eqref{eq:L1_der1}. As this is not the standard way to present convolution quadrature we give a detailed derivation of the scheme in the appendix. 

\begin{remark}
In order to understand better the formula \eqref{eq:cq_weights} it is of interest to compare  $\cqw_{n,j}$  with the weights $\omega_{n,j}$ of the L1 scheme.  Using the  approximation $(1-z)^{-1} = e^z+O(z^2)$ we have
\[
  \begin{split}    
 \cqw_{n,j} &= \tstep_j\frac1{2\pi i} \int_{\sigma-i \infty}^{\sigma+i\infty } z^{\beta-1} \prod_{k = j}^{n} \frac1{1-\tstep_kz} dz\\
&\approx \tstep_j\frac1{2\pi i} \int_{\sigma-i \infty}^{\sigma+i\infty } z^{\beta-1} e^{\sum_{k = j}^{n} \tstep_kz} dz\\
&= \tstep_j\frac1{2\pi i} \int_{\sigma-i \infty}^{\sigma+i\infty } z^{\beta-1} e^{(t_{n+1}-t_{j+1})z} dz\\
&= \tstep_j \frac1{\Gamma(1-\beta)}(t_{n+1}-t_j)^{-\beta},
  \end{split}
\]
where in the last step we used that the inverse Laplace transform of $z^{\beta-1}$ is $\frac1{\Gamma(1-\beta)} t^{-\beta}$. A rigorous proof in a more general setting and with error estimates is given in  \cite[Theorem~4.1]{Lub88I} for the case of uniform time-steps and in \cite[Proposition~2.2]{lbbook} for non-uniform time-steps.

Returning to the L1 discrete derivative \eqref{eq:omegan}
\[
  \begin{split}    
\omega_{n,j} &= \frac1{\Gamma(2-\beta)} \left((t_{n+1}-t_j)^{1-\beta}-(t_{n+1}-t_{j+1})^{1-\beta}\right)\\
&= \frac1{\Gamma(2-\beta)} \left((t_{n+1}-t_{j+1}+\tstep_j)^{1-\beta}-(t_{n+1}-t_{j+1})^{1-\beta}\right)\\
&= \tstep_j\frac{1}{\Gamma(1-\beta)} (t_{n+1}-t_{j+1})^{-\beta}+O((t_{n+1}-t_{j+1})^{-\beta-1}\tstep_j^2),
  \end{split}
\]
for $t_{n+1}> t_{j+1}$. Hence, the weights $\omega_{n,j}$ and $\cqw_{n,j}$ have a similar behaviour for $n \gg j$. 
\end{remark}

To simplify the computation of $\cqw_{n,j}$ we can transform the integration contour to the negative real axis 
\begin{equation}
  \label{eq:real_formula}
  \begin{split}    
  \cqw_{n,j}&= \tstep_j\frac1{2\pi i} \left[\int_{-\infty}^0e^{-(\beta-1)\pi i} x^{\beta-1} \prod_{k = j}^n \frac1{1+x\tstep_k} dx+\int_0^{\infty}e^{(\beta-1)\pi i} x^{\beta-1} \prod_{k = j}^n \frac1{1+x\tstep_k} dx\right]\\
&= \tstep_j\frac{\sin((1-\beta)\pi)}{\pi} \int_0^{\infty} x^{\beta-1} \prod_{k = j}^n \frac1{1+x\tstep_k} dx.
  \end{split}
\end{equation}
Note that
\[
\cqw_{n,n} = \tstep_n^{1-\beta}\frac{\sin((1-\beta)\pi)}{\pi} \int_0^{\infty} x^{\beta-1} \frac1{1+x} dx = \tstep_n^{1-\beta}.
\]
For $j < n$ we truncate the integral and then apply an $hp$-quadrature as done for the uniform time-step version in \cite{BanLop19}. This approach to computing the weights can be used to develop a computationally and memory efficient algorithm \cite{BanLop19}.

With this definition of the discrete derivative, the time-discrete system is given by
\begin{equation}
  \label{eq:CQ_disc}
[\cqd^{\beta} U ]_{n+1} + AU_{n+1} = f(t_{n+1}), \qquad n = 0,\dots, N-1,  
\end{equation}
or written as
\begin{equation}
  \label{eq:CQ_disc1}
\sum_{j = 0}^n \cqw_{n,j} \frac{U_{j+1}-U_j}{\tstep_j} + AU_{n+1} = f(t_{n+1}), \qquad n = 0,\dots, N-1.  
\end{equation}

For non-uniform steps, convolution quadrature has been analyzed in \cite{LopS16}. This analysis is however for hyperbolic problems with assumptions that are too stringent for the parabolic problem investigated here. For uniform steps \cite[Theorem~2.2]{Lub04} and the estimate $\left\|\left(A+ \lambda^\beta I\right)^{-1}\right\| \leq C |\lambda|^{-\beta}$, see \eqref{eq:resolvent}, imply that
\[
\|U_n-u(t_n)\| \leq C t_n^{\beta-1} \tstep, \qquad n = 1,\dots,N = T/\tstep.
\]
This implies that
\begin{align}  
\left(\tstep \sum_{n = 1}^N \|U_n-u(t_n)\|^2 \right)^{1/2}
&\leq C \tstep \left(\tstep \sum_{n = 1}^N t_n^{2\beta-2} \right)^{1/2}\nonumber\\
&\leq C  {\tstep \max(T^{\beta-1/2},\tstep^{\beta-1/2})}\nonumber\\
&\leq C(T) \tstep^{\min(1,\beta+\frac12)}\label{CQ_conv},
\end{align}
{where in the penultimate step we used an upper bound derived from the equivalent continuous integral; see \cite[Lemma~4.1]{BakerB}.} 
Therefore for $t_n$ away from the origin linear order convergence is obtained, whereas in the $L^2$ norm we expect to have suboptimal convergence order if  $\beta < 1/2$ unless a graded discretization is used.
\begin{remark}\label{rem:cq_l1_conv}
  We notice an important difference between the L1 scheme and CQ. The optimal convergence for the L1 scheme is $O(\tstep^{2-\beta})$, whereas for CQ  the optimal convergence order is linear $O(\tstep)$. For non-uniform time-step schemes we would expect the same to hold with $\tstep = \max_j \tstep_j$.
\end{remark}
\section{Stability of the continuous problem}\label{sec:apriori}

In this section we derive  stability estimates for the evolution problem \eqref{eq:sub_diff} in $L^2(H)$ and  $L^\infty (H), $ Theorem \ref{thm:stab_l2} and Theorem \ref{thm:stab_linf}. These bounds will be then instrumental to
derive the a posteriori estimates in the next section. We provide detailed proofs of both results aiming to   include a kernel in the forcing term. This is crucial in order to obtain an optimal estimator which can recognise the $O(\tstep^{2-\beta})$, i.e., higher than linear, convergence  order achievable by the L1 scheme. 

Important for the analysis in $L^2(H)$ will be the following positivity result. It can be deduced from Lemma~1.7.2 in \cite{siskova:phd}; see also Lemma~3.1 in \cite{EndreZener}. A similar result to the second inequality below can also be found in \cite[Theorem~A.1]{McLean:2012} but with a slightly less optimal constant. We nevertheless believe that the following proof is of interest. It is similar to  the  proof for the discrete stability given in \cite{SunWu}.

\begin{lemma}\label{lem:pos}
  Let $g \in C^1([0,T];H)$ and $\beta \in (0,1)$, where $H$ is a Hilbert space with inner product $\langle \cdot, \cdot \rangle$ and norm $\|\cdot\|$. Then
  \[
    \begin{split}      
    \int_0^T \langle\dt^\beta g(t), g(t)\rangle \, dt \geq& \frac{1/2}{\Gamma(1-\beta)} \int_0^T \left((T-t)^{-\beta}+t^{-\beta}\right)\|g(t)\|^2 \,dt \\&\qquad -\frac{1}{\Gamma(1-\beta)}\int_0^Tt^{-\beta}\langle g(0),g(t)\rangle\,dt\\
\geq& \frac{(T/2)^{-\beta}}{\Gamma(1-\beta)} \int_0^T \|g(t)\|^2 \,dt-\frac{1}{\Gamma(1-\beta)}\int_0^Tt^{-\beta}\langle g(0),g(t)\rangle\,dt.
    \end{split}
    \]
\end{lemma}
\begin{proof}
   To simplify notation, we set $H = \mathbb{R}$.  Let $\veps > 0$ and consider
\[
  \begin{split}    
    \int_0^T \int_0^t (t-\tau+\veps)^{-\beta} g'(\tau) g(t) d\tau dt =&
    -\beta\int_0^T \int_0^t (t-\tau+\veps)^{-\beta-1} g(\tau) g(t) d\tau dt\\
    &+ \veps^{-\beta}\int_0^T g^2(t)dt-g(0)\int_0^T(t+\veps)^{-\beta}g(t)dt.
  \end{split}
\]
Next we bound the first term on the right-hand side as follows
\[
\begin{split}
 \beta\int_0^T \int_0^t (t-\tau+\veps)^{-\beta-1} g(\tau) g(t) d\tau dt   \leq& \frac12\beta\int_0^T \int_0^t (t-\tau+\veps)^{-\beta-1} g^2(\tau) d\tau dt\\
        & +\frac12\beta\int_0^T \int_0^t (t-\tau+\veps)^{-\beta-1} g^2(t) d\tau dt\\       
        =& \frac12\beta\int_0^T \int_\tau^T (t-\tau+\veps)^{-\beta-1} dt\, g^2(\tau)  d\tau\\
        &+ \frac12\int_0^T (\veps^{-\beta}-(t+\veps)^{-\beta}) g^2(t) dt\\
 =& -\frac12\int_0^T ((T-\tau+\veps)^{-\beta}-\veps^{-\beta}) g^2(\tau)  d\tau\\
        &+ \frac12\int_0^T (\veps^{-\beta}-(t+\veps)^{-\beta}) g^2(t) dt.
\end{split}
\]
Returning to the original calculation, this inequality implies
\[
\begin{split}
    \int_0^T \int_0^t (t-\tau+\veps)^{-\beta} g'(\tau) g(t) d\tau dt  \geq& \frac12\int_0^T (T-\tau+\veps)^{-\beta} g^2(\tau)  d\tau+\frac12\int_0^T (t+\veps)^{-\beta} g^2(t) dt\\
        &-g(0)\int_0^T(t+\veps)^{-\beta}g(t)dt.
\end{split}
\]
Taking the limit $\veps \rightarrow 0^+$ and including the constant $\frac1{\Gamma(1-\beta)}$ gives the first inequality. Finding that $\min_{\tau \in [0,T]} \frac12((T-\tau)^{-\beta}+\tau^{-\beta}) = (T/2)^{-\beta}$ gives the second.
\end{proof}

%A corollary of this is a stability result for \eqref{eq:sub_diff}.

\begin{theorem}\label{thm:stab_l2}
  Let $u$ be the solution of \eqref{eq:sub_diff}. Then
  \[
    \begin{split}      
      2^{-1}\int_0^T g_{\beta,T}(t)&\|u(t)\|^2 \,dt+\int_0^T|u(t)|^2_1dt \\ &\leq \int_0^T \langle f(t),u(t)\rangle \,dt+\frac{1}{\Gamma(1-\beta)}\int_0^Tt^{-\beta}\langle u^0,u(t)\rangle\,dt,
\end{split}      
\]
where
\[
g_{\beta,T}(t) = \frac1{\Gamma(1-\beta)}\left((T-t)^{-\beta}+t^{-\beta}\right).
\]
\end{theorem}
\begin{proof}
  Testing \eqref{eq:sub_diff} with $u$ and using Lemma~\ref{lem:pos} gives the result.
\end{proof}

In the following corollary we estimate the forcing term by including a kernel dictated by the appearance of  $g_{\beta,T}$ in the lower bound of the estimate in Theorem \ref{thm:stab_l2}. This is in contrast to estimators of the forcing term in the dual of the $V$ norm typically appearing in diffusion problems. The next bound will be instrumental in the a posteriori analysis of the next section. 
\begin{corollary}\label{cor:stab_l2}
  Let $u$ be a solution of \eqref{eq:sub_diff}. Then
\[
  \begin{split}    
\int_0^T \|u(t)\|^2+|u(t)|_1^2\,dt \leq&  C^1_{T,\beta}\int_0^T (g_{\beta,T}(t))^{-1}\|f(t)\|^2 \,dt \\ &+C^2_{T,\beta}\|u^0\|^2
  \end{split}
\]
where $g_{\beta,T}$ as in Theorem~\ref{thm:stab_l2},
\[
C^1_{T,\beta} =   2 \max(2^{1-\beta}\Gamma(1-\beta)T^{\beta},1)
\]
and
\[
C^2_{T,\beta} =   2\frac{T^{1-\beta}}{\Gamma(2-\beta)}\max(2^{1-\beta}\Gamma(1-\beta)T^{\beta},1).
\]
\end{corollary}
\begin{proof}
Using the Cauchy-Schwarz and Young's inequalities in Theorem~\ref{thm:stab_l2} gives
\[
  \begin{split}    
2^{-1}(1-\epsilon_1-\epsilon_2)\int_0^T g_{\beta,T}(t) \|u(t)\|^2 \,dt &+\int_0^T|u(t)|_1\,dt \\
\leq& \frac12\epsilon_1^{-1}\int_0^T (g_{\beta,T}(t))^{-1}\|f(t)\|^2 \,dt \\ &+\frac12\epsilon_2^{-1}\frac1{\Gamma(1-\beta)}\int_0^T\frac1{t^{2\beta}\left((T-t)^{-\beta}+t^{-\beta}\right)}dt\|u^0\|^2\\
\leq& \frac12\epsilon_1^{-1}\int_0^T (g_{\beta,T}(t))^{-1}\|f(t)\|^2 \,dt +\frac12\epsilon_2^{-1}\frac{T^{1-\beta}}{\Gamma(2-\beta)}\|u^0\|^2.
      \end{split}
\]
Using $g_{\beta,T}(t) = \frac1{\Gamma(1-\beta)} \left((T-t)^{-\beta}+t^{-\beta}\right) \geq \frac2{\Gamma(1-\beta)}(T/2)^{-\beta}$ we have
\[
  \begin{split}    
\frac{(1-\epsilon_1-\epsilon_2)(T/2)^{-\beta}}{\Gamma(1-\beta)}\int_0^T \|u(t)\|^2 \,dt+\int_0^T|u(t)|_1\,dt \leq& \frac12\epsilon_1^{-1}\int_0^T (g_{\beta,T}(t))^{-1}\|f(t)\|^2 \,dt \\ &+\frac12\epsilon_2^{-1}\frac{T^{1-\beta}}{\Gamma(2-\beta)}\|u^0\|^2.
  \end{split}
\]
Setting $\epsilon_1 = \epsilon_2 = 1/4$ gives the result.
% \[
%   \begin{split}    
% \frac{2^{\beta-1}T^{-\beta}}{\Gamma(1-\beta)}\int_0^T \|u(t)\|^2 \,dt+\int_0^T|u(t)|_1\,dt \leq& 2\int_0^T (g_{\beta,T}(t))^{-1}\|f(t)\|^2 \,dt \\ &+2\frac{\Gamma(1-\beta)}{\beta+1} T^{\beta+1}\|u^0\|^2.
%   \end{split}
% \]
\end{proof}

Our next task is the   $L^\infty (H) $  stability and the proof of   Theorem \ref{thm:stab_linf} below. Towards this goal,  an  Abelian-Tauberian theorem with the names of Hardy and  Littlewood, and Karamata associated with it \cite[p.\ 445]{feller}, connects the asymptotic behaviour of a function $k(t)$ at $t = 0$ with the behaviour of its transform $K(s)$ at infinity. The more modest aim here is to give a bound on $k$ with an explicit constant.

\begin{lemma}\label{lem:kernel_bound}
Assume that  $K(s)$ is an analytic function  such that
\[
|K(s)| \leq C_K |s|^{-\mu} \qquad |\arg s| < \pi -\varphi,
\]
for some $\mu > 0$ and  $\varphi \in [0,\pi/2)$. Then  $k$, the inverse Laplace transform of $K$, is analytic for $t > 0$ and
\[
|k(t)| \leq C_KC_{t,\mu,\varphi}t^{\mu-1},
\]
where 
\[
  \begin{split}    
C_{t,\mu,\varphi} = \min_{r_0 \geq 0} \left(e^{r_0t} (r_0t)^{1-\mu}+\frac1{\pi}(\cos \varphi)^{\mu-1} \Gamma(1-\mu,r_0\cos \varphi)\right).
  \end{split}
\]

For $\mu \in (0,1)$, choosing $r_0 = 0$,  we have the explicit bound
\[
C_{t,\mu,\varphi} \leq \frac1{\pi}(\cos \varphi)^{\mu-1} \Gamma(1-\mu),
\]
whereas for $\mu = 1$, choosing $r_0 = t^{-1}$,
\[
C_{t,1,\varphi} \leq e+\frac1{\pi}E_1(t^{-1}\cos\varphi).
\]
\end{lemma}
\begin{proof}
We begin by using the inverse Laplace transform
\[
k(t) = \frac1{2\pi i}\int_\Gamma e^{st}K(s)ds 
\]
to represent $k$. Here $\Gamma$ consists of the two arcs $\Gamma_{-}\colon s = r e^{(-\pi+\varphi)i}$ and $\Gamma_+\colon s = r e^{(\pi-\varphi)i}$, for $r \in [r_0,\infty)$ and the circular contour connecting these: $\Gamma_0 = \{ r_0e^{i\varphi} : \varphi \in (-\pi+\varphi,\pi-\varphi)\}$ where $r_0 \in (0,1]$.   Starting with $\Gamma_0$ we have
\[
\left|\frac1{2\pi i}\int_{\Gamma_0}e^{st}K(s)ds\right| \leq C_Ke^{r_0t}r_0^{1-\mu}.
\]
The bound on $\Gamma_+ \cup \Gamma_-$ can be done at the same time
\[
  \begin{split}    
\left|    \frac1{2\pi i}\int_{\Gamma_-\cup \Gamma_+}e^{st}K(s)ds\right| &\leq \frac{C_K}{\pi} \int_{r_0}^\infty r^{-\mu} e^{-rt \cos \varphi}dr\\
%&= \frac{C_K}{\pi}t^{\mu-1} \int_{r_0}^\infty r^{-\mu} e^{-r \cos \varphi}dr\\
&= \frac{C_K}{\pi}t^{\mu-1}(\cos \varphi)^{\mu-1} \int_{r_0 \cos \varphi}^\infty r^{-\mu} e^{-r }dr\\
&= \frac{C_K}{\pi}t^{\mu-1}(\cos \varphi)^{\mu-1} \Gamma(1-\mu,r_0\cos \varphi),
%&\leq \frac{C_K}{\pi}t^{\mu-1}(\cos \varphi)^{\mu-1} \Gamma(1-\mu),
  \end{split}
\]
where $\Gamma(\cdot,\cdot)$ is the incomplete Gamma function. For $\mu \in (0,1)$ we set $r_0 = 0$ 
\[
|k(t)| \leq \frac{C_K}{\pi}t^{\mu-1}(\cos \varphi)^{\mu-1} \Gamma(1-\mu,0)
= \frac{C_K}{\pi}t^{\mu-1}(\cos \varphi)^{\mu-1} \Gamma(1-\mu).
\]
For $\mu = 1$, we choose $r_0 = t^{-1}$. Using 
$
\Gamma(0, r) = E_1(r)
$
\cite[(6.11.1)]{NIST}, where $E_1(\cdot)$ is the exponential integral \cite[(6.2.1)]{NIST}, we obtain that
\[
|k(t)| \leq C_K (e+\frac1{\pi}E_1(t^{-1}\cos\varphi)).
\]
\end{proof}

\begin{remark}
  Note that $E_1(t^{-1}\cos \varphi)$ grows as $O(\log t^{-1})$ for $t \rightarrow \infty$.
\end{remark}
%Using this, we can prove the following pointwise bound on the solution.

Combining the above lemma with a resolvent bound for $A$ will give another stability bound. {As $A$ is a positive definite, self-adjoint operator its spectrum lies on the positive real axis, and hence from \cite[V (3.16)]{Kato} we have} 
 the  resolvent bound 
\begin{equation}
  \label{eq:resolvent}
  \|(A+\lambda I)^{-1}\| {= \frac1{\operatorname{dist}(\lambda,\sigma(A))}} \leq \frac{1}{\sin \theta |\lambda|} \quad \text{ for } |\arg \lambda| < \pi-\theta
\end{equation}
and any fixed $\theta \in (0,\pi/2)$.

\begin{theorem}\label{thm:stab_linf}
Let $u$ be the solution of \eqref{eq:sub_diff} with the additional assumption $u^0 \in D(A)$. Then the bound
\[
\|u(t)\| \leq \|u^0\|+\frac1{\sin \theta} C_{\beta,\varphi}\int_0^t(t-\tau)^{\beta-1}\|f(\tau)-Au^0\|d\tau
\]
holds for $t > 0$, where $\varphi = \max(0,\pi-(\pi-\theta)/\beta)$ and
\[
C_{\beta,\varphi} =  \frac1{\pi}(\cos \varphi)^{\beta-1} \Gamma(1-\beta).
\]
\end{theorem}
\begin{proof}
Writing $w = u-u^0$ \eqref{eq:sub_diff} becomes
\[
\dt^{\beta}w+Aw = f-Au^0.
\]
After taking the Laplace transform  we obtain that
\[
\hat w(s) = (s^\beta{+}A)^{-1}\left(\hat f(s)-\frac1{s} Au_0\right),
\]
where $\hat w$ and $\hat f$ are the Laplace transforms of $w$ and $f$; note that the growth condition on $f$ ensures the existence of the Laplace transform.

Thus, we have  the representation 
\begin{equation}
  \label{eq:representation_formula}  
u(t) = u^0+\int_0^t k_1(t-\tau) (f(\tau)-Au^0)d\tau,
\end{equation}
with
\[
K_1(s) = (s^\beta{+}A)^{-1}.
\]
The resolvent bound \eqref{eq:resolvent} {implies
\[
\|K_1(s)\| \leq \frac1{\sin \theta |s|^\beta},
\]
for $|\arg s^\beta| < \pi -\theta$ or equivalently for $|\arg s| < (\pi -\theta)/\beta$. Combining this with}
Lemma~\ref{lem:kernel_bound} and $\varphi = \max(0,\pi-(\pi-\theta)/\beta)$ gives that
\[
\|u(t)\|
\leq \|u^0\| +\frac1{\sin \theta} C_{\beta,\varphi}\int_0^t(t-\tau)^{\beta-1}\|f(\tau)-Au^0\|d\tau.
%  \end{split}
\]
\end{proof}

\section{A posteriori error analysis: piecewise linear reconstruction}\label{sec:apost}

Recall that $\hat U$ denotes the  continuous piecewise linear interpolant of the data $U_0, U_1, \dots, U_N$.  Due to \eqref{eq:discrete_sys}, \eqref{eq:CQ_disc} both schemes can be written in the form
 \begin{equation}
      \label{pw_1}
      Q^\beta \hat U(t_{n+1})+A U_{n+1} =   Q^\beta \hat U(t_{n+1})+A \hat U(t_{n+1}) = f(t_{n+1}),
    \end{equation}
  where for  $ n = 0,\dots, N-1,$
  \begin{equation}
      \label{pw_2}
Q^\beta \hat U(t_{n+1}) = \sum_{j = 0}^{n} \int_{I_j} K_{n,j} \hat U ' (\tau) d\tau  \quad \text{where }   K_{n,j} = \begin{cases}   \omega_{n,j} / {\tstep_j}&\quad \text {for L1 scheme},\\
\cqw_{n,j} / {\tstep_j} &\quad \text {for C-Q scheme.}  \end{cases}
\end{equation}
Furthermore, for consistency reasons we set $Q^\beta \hat U(0)=0. $ 
We need a piecewise equation to be valid for all $t.$ In contrast to the parabolic problems where $ \hat U ' (t_{n+1}) =\hat U ' (t), $ for all $t\in I_n,$ and thus the pointwise equation can be extended for all $t,$ in our case we should proceed  in a different way. In fact, we derive a pointwise equation for $\hat U$ by applying  piecewise linear interpolation to \eqref {pw_1}. Then we conclude, 
\begin{equation}
  \label{eq:Uhat_eqn}  
\Pi_1 Q^{\beta} \hat U(t)+ A\hat U(t) = \Pi_1 f(t)+\hat G(t), \qquad t \in [0,T],
\end{equation}
with initial data $\hat U(0) = U_0$, where $\hat G$ is a piecewise linear correction term whose role we explain next. {Here, $\Pi_1 Q^{\beta} \hat U$ is the piecewise linear interpolant of the function $Q^\beta \hat U$ defined in \eqref{pw_2}.} By definition, \eqref{eq:Uhat_eqn} is satisfied for $t = t_j$, $j = 1,\dots,N$ if $\hat G(t_j) = 0$. However, as 
we have set $Q^\beta \hat U(0) = 0$, in order that the equation is satisfied also in the first interval we need that $\hat G(0) = AU_0-f(0)$. Hence, the piecewise linear correction function $\hat G$ is defined by interpolating
\begin{equation}
  \label{eq:Ghat}  
\hat G(0) = AU_0-f(0), \quad \hat G(t_j) = 0, \; j = 1,\dots,N.
\end{equation}
We conclude therefore that  $\hat U$ is the solution of the original evolution problem \eqref{eq:sub_diff} with a modified right-hand side
\begin{equation}
      \label{pw_final}
  \dt^{\beta} \hat U(t)+ A\hat U(t)  = \Pi_1 f(t)+\hat G(t)+\dt^{\beta} \hat U(t)- \Pi_1 Q^{\beta} \hat U(t), \qquad t \in (0,T].
 \end{equation}

%%%%%%%%%%L 1
%We can now write down the fully discrete system: Find $U_{n+1} \in H$, $n = 0,1,\dots,N-1$ such that
%    \begin{equation}
%      \label{eq:discrete_sys}
%      \Ld U(t_{n+1})+AU_{n+1} = f_{n+1},
%    \end{equation}
%    where $f_{n+1} = f(t_{n+1})$ and $U_0 = u^0$ or some approximation of the initial data. Alternatively, recalling the definition of $\omega_{n,j}$ \eqref{eq:omegan}  we can rewrite the system  in a more familiar form as a finite difference formula
%    \begin{equation}
%      \label{eq:discrete_sys1}
%\sum_{j = 0}^{n} \omega_{n,j} \frac1{\tstep_j}(U_{j+1}-U_j) + AU_{n+1} = f_{n+1}, \qquad n = 0,\dots, N-1.
%\end{equation}
%
%
%%%%%%%%%%%C Q 
% With this definition of the discrete derivative, the time-discrete system is given by
%\begin{equation}
%  \label{eq:CQ_disc}
%[\cqd^{\beta} U ]_{n+1} + AU_{n+1} = f(t_{n+1}), \qquad n = 0,\dots, N-1  
%\end{equation}
%or written as
%\begin{equation}
%  \label{eq:CQ_disc1}
%\sum_{j = 0}^n \cqw_{n,j} \frac{U_{j+1}-U_j}{\tstep_j} + AU_{n+1} = f(t_{n+1}), \qquad n = 0,\dots, N-1.  
%\end{equation}
%
%
%
%
%

\subsection{Error equation} The error  $e = u-\hat U$ is then the solution of
\[
\dt^\beta e+Ae = f-\dt^{\beta} \hat U- A\hat U=: \hat R, \qquad e(0) = u^0-U_0,
\]
where in view of \eqref{pw_final} the residual $\hat R$ is given by 
\begin{equation}
\hat R   = f (t) - \Pi_1 f(t)-\hat G(t)- (\dt^{\beta} \hat U(t)- \Pi_1 Q^{\beta} \hat U(t)), \qquad t \in (0,T].
 \end{equation}

The continuous stability of the fractional problem  \eqref{eq:sub_diff} implies the desired a posteriori bounds. The stability estimates in Corollary~\ref{cor:stab_l2} and Theorem~\ref{thm:stab_linf} imply two different a posteriori error estimates. 

\begin{theorem}\label{thm:apost_l2}
  Let $u$ be the solution of \eqref{eq:sub_diff} and $U_n$, $n = 0,\dots,N$, the solution of the discrete system \eqref{eq:discrete_sys} (or \eqref{eq:discrete_sys_corr}). The error $ e = u-\hat U$ satisfies the bound
\[
  \begin{split}    
\int_0^t  \| e(\tau)\|^2+| e(\tau)|_1^2\,d\tau \leq& C_{t,\beta}^1 \int_0^t(g_{\beta,t}(\tau))^{-1}\left\|f(\tau)- \dt^\beta \hat U(\tau)-A \hat U(\tau)\right\|^2\,d\tau\\
&+C^2_{t,\beta}\|u^0-U_0\|^2,
  \end{split}
\]
for $t >0$, and $g_{\beta,t}$, $C^1_{t,\beta}$, $C^2_{t,\beta}$ as in Theorem~\ref{thm:stab_l2} and Corollary~\ref{cor:stab_l2}
\end{theorem}

\smallskip
\noindent
The next estimator controls the error in $L^\infty (H)\, .$

\begin{theorem}\label{thm:apost_linf}
  Let $u$ be the solution of \eqref{eq:sub_diff} and $U_n$, $n = 0,\dots,N$, the solution of the discrete system \eqref{eq:discrete_sys} (or \eqref{eq:discrete_sys_corr}). The error $ e = u-\hat U$ satisfies the bound
\[
\| e(t)\| \leq \|u^0-U_0\|+\frac1{\sin \theta}  C_{\beta,\varphi}\int_0^t(t-\tau)^{\beta-1}\left\|f(\tau)- \dt^\beta \hat U(\tau)-A \hat U(\tau)\right\|d\tau,
\]
for any $\theta \in (0,\pi/2)$, $t \in (0,T]$, with $\varphi = \max(0,\pi-(\pi-\theta)/\beta)$,  and 
\[
C_{\beta,\varphi} =  \frac1{\pi}(\cos \varphi)^{\beta-1} \Gamma(1-\beta).
\]
\end{theorem}

Next we investigate the optimality of the above a posteriori error estimators.

\subsection{Asymptotic behaviour  of the estimators}
In the extended numerical experiments presented in the last section we demonstrate that both estimators accurately capture the asymptotic behaviour of the error in several cases.
In the rest of this section we consider the question of asymptotic behaviour  of the a posteriori estimators under certain assumptions on the solution. We are particularly interested whether the estimator can converge at the optimal convergence order of the L1 scheme $O(\tstep^{2-\beta})$, i.e., better than linear. As a proof of concept, we just consider the estimator  in Theorem \ref{thm:apost_linf}
for the L1 scheme. Given   the presence of the kernels in the estimators and the complicated a priori analysis required, the discussion of the other cases is left for a future work. 
For the L1 scheme $\hat U$ is the solution of
\begin{equation}
  \label{eq:Uhat_eqn_L1}  
\Pi_1 \dt^{\beta} \hat U(t)+ A\hat U(t) = \Pi_1 f(t)+\hat G(t), \qquad t \in [0,T],
\end{equation}
with initial data $\hat U(0) = U_0$, where $\hat G$ is the piecewise linear correction term introduced above.
The approximations $\hat U$ satisfy the original problem \eqref{eq:sub_diff} with a modified right-hand side
\[
  \dt^{\beta} \hat U(t)+ A\hat U(t) = \Pi_1 f(t)+\hat G(t)+\dt^{\beta} \hat U(t)- \Pi_1\dt^{\beta} \hat U(t), \qquad t \in (0,T],
\]
and  the error  $e = u-\hat U$ solves 
\begin{equation}
  \label{eq:err_eqn}
  \dt^{\beta} e(t)+ A e(t) = f(t)-\Pi_1 f(t)-\hat G(t)+\Pi_1\dt^{\beta} \hat U(t)- \dt^{\beta} \hat U(t), \qquad t \in (0,T],
\end{equation}
and $ e(0) = u^0-U_0$. 

% \begin{remark}
%   To treat the corrected scheme $\hat G(t)$ needs to be changed to  the piecewise linear function defined by
%   \begin{equation}
%     \label{eq:Ghat_corr}
% \hat G(0) = Au^0-f(0),\; \hat G(t_1) = -\frac12(Au^0-f(0)),\; \hat G(t_j) = 0, \; j = 2,\dots,N.    
%   \end{equation}
% \end{remark}

First of all let us investigate the effect of $\hat G$.

\begin{lemma}
  Let $\hat G$ be the correction function defined by \eqref{eq:Ghat}. Then for $t > \tstep_0$
\[
\int_0^t (g_{\beta,t}(\tau))^{-1}\|\hat G(\tau)\|^2 d\tau \leq C \|AU_0-f(0)\|^2 \tstep_0^{1+\beta},
\]
and
\[
  \begin{split}
\int_0^t (t-\tau)^{\beta-1} \|\hat G(\tau)\|d \tau &\leq \frac1{\beta}\|AU_0-f(0)\| \left(t^\beta-(t-\tstep_0)^\beta\right)\\
&\leq C\tstep_0 \|AU_0-f(0)\| (t{-\tstep_0})^{\beta-1},
  \end{split}
\]
for some constant $C > 0$ depending on $\beta$.
\end{lemma}
\begin{proof}
The first inequality follows from 
  \[
  \begin{split}    
\int_0^t (g_{\beta,t}(\tau))^{-1}\|\hat G(\tau)\|^2 d\tau &\leq 
\Gamma(1-\beta)\|AU_0-f(0)\|^2\int_0^{\tstep_0} ((t-\tau)^{-\beta}+\tau^{-\beta})^{-1}d\tau\\
& \leq \frac1{\beta+1} \Gamma(1-\beta)\|AU_0-f(0)\|^2 \tstep_0^{1+\beta},
  \end{split}
\]
{where we used
\[
\int_0^{\tstep_0} ((t-\tau)^{-\beta}+\tau^{-\beta})^{-1}d\tau
\leq \int_0^{\tstep_0} \tau^{\beta}d\tau \leq \frac1{1+\beta}\tstep_0^{\beta+1}.
\]
}
The second inequality is obtained by direct computation and the estimate
\[
\frac1{\tstep_0}\left(t^\beta-(t-\tstep_0)^{\beta}\right) \leq \beta \max_{\tau \in [t-\tstep_0,t]} \tau^{\beta-1} = \beta (t-\tstep_0)^{\beta-1}.
\]
\end{proof}
\begin{remark}\label{rem:hatG_error}
  This suggests that in general the $L^2$ norm of the error is no better than $O(\tstep_0^{\frac{1+\beta}{2}})$ and the $L^\infty$ norm no better than $O(t^{\beta-1}\tstep_0)$. Hence, denoting by $\tstep = \max_j \tstep_j$,   we require at least that $\tstep_0 = O(\tstep^{\frac{4-2\beta}{1+\beta}})$ in order to obtain optimal  assymptotic convergence order $O(\tstep^{2-\beta})$ in the $L^2$ norm. In the $L^\infty$ norm, we require at least $\tstep_0 = O(\tstep^{2-\beta})$ to obtain optimal convergence for $t_n > c$ for any fixed constant $c >0$. Further,  $\tstep_0 = O(\tstep^{\frac{2-\beta}{\beta}})$ is required if optimal convergence is to be expected uniformly for $t_n > \tstep_0$; this is compatible with the result from \cite{StynesOG}; see \eqref{eq:stynes_graded}.
\end{remark}

If the data $f$ is smooth for $t >0$ then so is the solution $u$ with a possible singularity at $t = 0$; see \eqref{eq:representation_formula}.

\begin{lemma}\label{lem:L1_err}
Let $u \in {C[0,T] \cap} C^2(0,T]$ {and denote by $\hat u$ the piecewise linear interpolant of the data $u(t_j)$, $j = 0,\dots,n+1$ where $0 = t_0 < t_1 < \dots < t_{n+1}$}. Then there exist $z_0 \in (0,t_1)$ and $\tilde z_j \in (t_{j-1},t_{j+1})$, $j = 1,\dots,n$,  such that for $t \in I_n$, $n \geq 1$,
\[
  \begin{split}    
\left|\Pi_1 \dt^\beta \hat u(t) - \dt^\beta \hat u(t)\right|
\leq& C \Bigl(\tstep_0^2 t^{-1-\beta}|u'(z_0)|\\ &+\sum_{j = 1}^{n-1} \tstep_j^2(\tstep_{j-1}+\tstep_j)(t-t_j)^{-\beta-1}|u''(\tilde z_j)|+(\tstep_{n-1}+\tstep_n)\tstep_n^{1-\beta}|u''(\tilde z_n)|\Bigr).
  \end{split}
\]
For $t \in I_0$
\[
\left|\Pi_1 \dt^\beta \hat u(t) - \dt^\beta \hat u(t)\right| \leq C \tstep_0^{1-\beta} |u'(z_0)|.
\]
\end{lemma}
\begin{proof}
  First note that there exist $z_j \in (t_j,t_{j+1})$ and $\tilde z_j \in (z_j,z_{j+1})$  such that
\[
\frac{u(t_1)-u(0)}{\tstep_0} = u'(z_0)
\]
and
\[
  \begin{split}    
\frac{u(t_{j+1})-u(t_j)}{\tstep_j}-\frac{u(t_{j})-u(t_{j-1})}{\tstep_j}
&= u'(z_j)-u'(z_{j-1})\\  &= (z_j-z_{j-1})u''(\tilde z_j).
  \end{split}
\]
Using the representation \eqref{eq:upp_repr} we hence have for $t \leq t_{n+1}$
\[
  \begin{split}    
\Pi_1 \dt^\beta \hat u(t) - \dt^\beta \hat u(t)
=& (\Pi_1 a_0(t)-a_0(t))u'(z_0)\\ &+\sum_{j = 1}^n (\Pi_1 a_j(t)-a_j(t)) (z_j-z_{j-1})u''(\tilde z_j).
  \end{split}
\]
For $t \in I_n$ and $j \leq n-1$ we have
\[
|\Pi_1 a_j(t)-a_j(t)| \leq C \tstep_n^2 (t-t_j)^{-1-\beta}.
\]
Whereas for $t \in I_n$
\[
|\Pi_1 a_n(t)-a_n(t)| \leq C \tstep_n^{1-\beta}.
\]
Combining the last three statements gives the first result, whereas the last statement for $n = 0$ gives the  second required result.
\end{proof}

% For a smooth $f$, we expect $u$ to be smooth for $t > 0$ and bounded as
% \begin{equation}
%   \label{eq:uder_est}
% |\partial_t^k u| \leq C (1+t^{\beta-k})  
% \end{equation}
% for $t$ close to $0$.

If the solution $u$ is smooth, namely $u \in C^2[0,T]$ it is shown in \cite{LinXu,SunWu} that for uniform time-steps $\tstep = \tstep_j$ optimal convergence order
\[
\|u(t_n)-U_n\| \leq C \tstep^{2-\beta}
\]
is obtained.
We investigate now if our estimator in Theorem~\ref{thm:apost_linf} achieves this.

First of all we note that if $u$ is continuously differentiable, then $\dt^\beta u(0) = 0$ and consequently \eqref{eq:sub_diff} implies $Au_0 = f(0)$. Therefore if we take as initial data $U_0 = u^0$, the correction function vanishes $\hat G \equiv 0$. Next, for $u$ to be smooth, in general  $f$ is not smooth but behaves as
\[
f(t) \sim f(0)+c_1 t^{1-\beta}
\]
asymptotically as $t \rightarrow 0$. To understand the origin of this singularity, simply substitute $u = u(0)+u'(0) t$ in the fractional equation  \eqref{eq:sub_diff}.

Hence, in order to investigate the term due to $f$ in the a posteriori error estimate we need the following lemma.

\begin{lemma}
Let $\beta \in (0,1)$. Then for a uniform mesh and $n \geq 1$
\[
\int_0^{t_{n+1}} (t_{n+1}-\tau)^{\beta-1} |\Pi_1 \tau^{1-\beta} -\tau^{1-\beta}|d\tau \leq C \tstep^{2-\beta} (t_{n+1}-\tstep)^{\beta-1}.
\]
\end{lemma}
\begin{proof}
By results on linear interpolation we have that
   \[
|\Pi_1 t^{1-\beta} -t^{1-\beta}| \leq \frac12 (1-\beta)\beta \tstep^2 t_n^{-\beta-1} \qquad t \in I_n
\]
and  for $t \in I_0$
 \[
   \begin{split}     
|\Pi_1 t^{1-\beta} -t^{1-\beta}| &\leq {\max_{t \in [0,\tstep]} |\tstep^{-\beta}t -t^{1-\beta}|}\\
&{= }   (1-\beta)^{\frac{1}{\beta}}((1-\beta)^{-1}-1)\tstep^{1-\beta}.
   \end{split}
\]
Hence for $t \in I_n = (t_n,t_{n+1}]$
\[
  \begin{split}    
\int_0^{t_{n+1}} (t_{n+1}-\tau)^{\beta-1} |\Pi_1 \tau^{1-\beta} -\tau^{1-\beta}|d\tau \leq& C\tstep^{1-\beta} \int_0^\tstep (t_{n+1}-\tau)^{\beta-1}d\tau \\
&+ C\tstep^2 \sum_{j = 1}^{n}t_j^{-\beta-1}\int_{t_j}^{t_{j+1}} (t_{n+1}-\tau)^{\beta-1}d\tau\\
\leq& C\tstep^{2-\beta} (t_{n+1}-\tstep)^{\beta-1} \\
&+ C\tstep^3 \sum_{j = 1}^{n-1}t_j^{-\beta-1}(t_{n+1}-t_{j+1})^{\beta-1}\\ 
&+C\tstep^{2+\beta} t_n^{-\beta-1} \\
=& C\tstep^{2-\beta} (t_{n+1}-\tstep)^{\beta-1} 
+ C\tstep \sum_{j = 1}^{n-1}j^{-\beta-1}(n-j)^{\beta-1}\\ 
&+C\tstep^{2+\beta} t_n^{-\beta-1} \\
\leq& C\tstep^{2-\beta} (t_{n+1}-\tstep)^{\beta-1} 
+ C\tstep^{2-\beta} t_n^{\beta-1}\\
&+C\tstep^{2+\beta} t_n^{-\beta-1}.
  \end{split}
\]
In the last step we used the following estimate from \cite[Lemma~5.3]{Lub88I}
\[
\sum_{j = 1}^{n-1}j^{-\beta-1}(n-j)^{\beta-1} = O(n^{\beta-1}).
\]
  As $\tstep^{2+\beta}t_n^{-\beta-1} \leq C \tstep^{2-\beta}t_n^{\beta-1}$ for $t_n > \tstep$ we have the required result.
\end{proof}

Finally we investigate the term due to $\Pi_1 \dt^\beta \hat U(\tau) - \dt^\beta \hat U(\tau)$.

\begin{lemma}
  Let $u \in C^2[0,T]$, then for $t > \tstep$
\[
\int_0^t (t-\tau)^{\beta-1}\left|\Pi_1 \dt^\beta \hat u(\tau) - \dt^\beta \hat u(\tau)\right| d\tau  \leq C \tstep^{2-\beta} (t-\tstep)^{\beta-1}.
\]
\end{lemma}
\begin{proof}  
For $t \in I_n$ and $n > 0$, Lemma~\ref{lem:L1_err} implies
\[
  \begin{split}    
\left|\Pi_1 \dt^\beta \hat u(t) - \dt^\beta \hat u(t)\right|
&\leq C \left(\tstep^2 t^{-1-\beta}+ \tstep^3\sum_{j = 1}^{n-1}(t-t_j)^{-\beta-1}+\tstep^{2-\beta}\right)\\
&\leq C\left(\tstep^2 t^{-1-\beta}+\tstep^2\int_0^{t_{n-1}}(t-\tau)^{-\beta-1}d\tau+\tstep^{2-\beta}\right)\\
&\leq C\left(\tstep^2 t^{-1-\beta}+\tstep^{2-\beta}\right).
  \end{split}
\]
For $t \in I_0$
\[
\left|\Pi_1 \dt^\beta \hat u(t) - \dt^\beta \hat u(t)\right| \leq C \tstep^{1-\beta}.
\]
Hence
\[
  \begin{split}
    \int_0^t (t-\tau)^{\beta-1}\left|\Pi_1 \dt^\beta \hat u(\tau) - \dt^\beta \hat u(\tau)\right| d\tau \leq& C \left(\tstep^{1-\beta} \int_0^\tstep (t-\tau)^{\beta-1}d\tau\right.\\ &+\tstep^{2-\beta}\int_\tstep^t (t-\tau)^{\beta-1}d\tau\\
&+ \left. \tstep^{2}\int_\tstep^t \tau^{-1-\beta}(t-\tau)^{\beta-1}d\tau\right)\\
\leq& C \left(\tstep^{1-\beta} (t^\beta-(t-\tstep)^{\beta})\right.\\
&+ \left. \tstep^{2}\int_\tstep^t \tau^{-1-\beta}(t-\tau)^{\beta-1}d\tau\right)\\
\leq& C \tstep^{2-\beta} (t-\tstep)^{\beta-1}.
  \end{split}
\]
\end{proof}

\section{Numerical experiments}\label{sec:numerics}

\subsection{Fractional differential equation with $f \equiv 0$}
First we consider the simple fractional ordinary differential equation
\begin{equation}
  \label{eq:ode_0}  
\dt^\beta u+\lambda u = 0, \qquad u(0) = u^0, \qquad t \in (0,T),
\end{equation}
with $\lambda > 0$ a fixed constant.

The solution is given by $u(t) = E_\beta(-\lambda t^\beta)u^0$ where
\[
 E_\beta(z) = \sum_{m = 0}^\infty \frac{z^m}{\Gamma(\beta m+1)}
\]
is the Mittag--Leffler function. Note that $u \sim (1-\frac1{\Gamma(\beta+1)} t^{\beta})u^0$  for $t \rightarrow 0$, hence already the first derivative of $u$ is unbounded at $t = 0$.

We will compare the  convergence of the L1 scheme and CQ with our estimators. The meshes will be of the form
\[
t_j = T (j/N)^k
\]
with $k \geq 1$; the mesh is uniform if $k = 1$ and graded towards 0 if $ k > 1$. The exact initial data will be used, i.e., $U_0 = u^0$. 

The estimators in Theorem~\ref{thm:apost_l2} and Theorem~\ref{thm:apost_linf} require the numerical computation of the outer integral, which we throughout this section compute using a compound midpoint rule. The inner term $\dt^\beta\hat U$ is just the L1 fractional derivative and is computed exactly.

We denote by
\begin{equation}
  \label{eq:E1}  
E^1(t) \approx \left(\int_0^t |u(\tau) -\hat U(\tau)|^2 d\tau\right)^{1/2}
\end{equation}
the exact error in $L^2$ norm approximated by the compound midpoint rule and by
\begin{equation}
  \label{eq:E2}  
E^2(t) = |u(t) -\hat U(t)|
\end{equation}
the exact error at time $t$. The estimators, again approximated by the compound midpoint rule, are denoted by $E^1_{\text{est}}(t)$ and $E^2_{\text{est}}(t)$ corresponding to Theorem~\ref{thm:apost_l2} and Theorem~\ref{thm:apost_linf} respectively.

\begin{table}
  \centering
\subfloat[Uniform mesh with $\beta = 0.8$.]{
  \begin{tabular}{|c|c|c|c|c|}
\hline
 $N$ &  $E^1$ & eoc & $E^1_{\text{est}}$ & eoc\\\hline
10 & $0.01$ &  & $0.054$ & \\
20 & $0.0055$ & 0.89 &$0.029$ & 0.92\\
40 & $0.0029$ & 0.91 &$0.015$ & 0.93\\
80 & $0.0016$ & 0.92 &$0.0080$ & 0.93\\
160 &$0.00081$ &0.94  &$0.0042$ & 0.92\\
320 &$0.00042$ &0.95  &$0.0022$ & 0.92\\\hline
\end{tabular}
}
\subfloat[Graded mesh with $k = 2$ and $\beta = 0.8$.]{
\begin{tabular}{|c|c|c|c|c|}
\hline
  $N$ &  $E^1 $ & eoc & $E^1_{\text{est}}$ & eoc\\\hline
10 &  $0.0069$ &  & $0.042$ & \\
20 &  $0.0032$ & 1.13 &$0.019$ & 1.15\\
40 &  $0.0014$ & 1.16 &$0.0083$ & 1.18\\
80 &  $0.00064$ & 1.17 &$0.0036$ & 1.19\\
160 & $0.00028$ &1.19  &$0.0016$ & 1.20\\
320 & $0.00012$ &1.19  &$0.00069$ & 1.20\\\hline
\end{tabular}
}

\subfloat[Uniform mesh with $\beta = 0.2$.]{
\begin{tabular}{|c|c|c|c|c|}
\hline
 $N$ &  $E^1$ & eoc & $E^1_{\text{est}}$ & eoc\\\hline
$10$ & $0.01$ &  & $0.11$ &  \\
$20$ & $0.0071$ & $0.56$ & $0.076$ & $0.57$ \\
$40$ & $0.0048$ & $0.57$ & $0.051$ & $0.57$ \\
$80$ & $0.0032$ & $0.59$ & $0.034$ & $0.58$ \\
$160$ & $0.0021$ & $0.6$ & $0.023$ & $0.58$  \\\hline
\end{tabular}
}
%\subfloat[Graded mesh with $k = 2$ and $\beta = 0.2$.]{
\subfloat[Graded mesh with $k = 3$ and $\beta = 0.2$.]{
\begin{tabular}{|c|c|c|c|c|}
\hline
 $N$ &  $E^1$ & eoc & $E^1_{\text{est}}$ & eoc\\\hline
$10$ & $0.0011$ &  & $0.012$ &  \\
$20$ & $0.00037$ & $1.6$ & $0.0038$ & $1.63$ \\
$40$ & $0.00012$ & $1.64$ & $0.0012$ & $1.65$ \\
$80$ & $3.8\times 10^{-5}$ & $1.67$ & $0.00038$ & $1.67$ \\
$160$ & $1.2\times 10^{-5}$ & $1.7$ & $0.00012$ & $1.69$ \\
$320$ & $3.5\times 10^{-6}$ & $1.71$ & $3.6\times 10^{-5}$ & $1.7$ \\\hline
%$640$ & $1.1\times 10^{-6}$ & $1.7$ & $1.1\times 10^{-5}$ & $1.7$   
%$10$ & $0.003$ &  & $0.032$ & $0$ \\
%$20$ & $0.0013$ & $1.2$ & $0.014$ & $1.2$ \\
%$40$ & $0.00056$ & $1.2$ & $0.0064$ & $1.2$ \\
%$80$ & $0.00024$ & $1.2$ & $0.0028$ & $1.2$ \\
%$160$ & $9.8\times 10^{-5}$ & $1.3$ & $0.0013$ & $1.2$ \\\hline
\end{tabular}
}
  \caption{The numerical results for the L1 scheme including the error measured in the $L^2$ norm at time $t = 1/2$ and the estimated order of convergence for differerent choices of $\beta$ and with uniform and graded meshes.}
  \label{tab:L2_L1}
\end{table}

In Table~\ref{tab:L2_L1}  we show the results for the L1 scheme using the $L^2$ error measure \eqref{eq:E1} evaluated at $t = 1/2$ and for two values of the parameter $\beta \in \{0.2,0.8\}$.  
We see that suboptimal convergence order is obtained using the uniform mesh, but the optimal convergence order is recovered with the graded mesh. In both cases the estimator converges at the correct order.  Remark~\ref{rem:hatG_error} predicts  the requirement for the graded mesh $k = \frac{4-2\beta}{1+\beta}$, i.e., $k = 4/3$ for $\beta = 0.8$ and $k = 3$ for $\beta = 0.2$. {For $\beta = 0.8$ we see that for $k = 2 > 4/3$ optimal convergence is obtained.} For $k = 4/3$ and $\beta = 0.8$ numerical experiments suggest that asymptotically the optimal convergence is achieved but only for quite small time steps; see Table~\ref{tab:optimal_k}. { For $\beta = 0.2$, we see that the observed rate is slowly approaching the optimal rate 1.8 when the borderline grading $k = 3$ is used.} Numerical experiments not reported here for $k$ smaller than $k = 3$ for $\beta = 0.2$ give less than optimal convergence order.

\begin{table}
  \centering
  \begin{tabular}{|c|cccc|}\hline
    $N$& 160 & 320 & 640 & 1280\\\hline
    eoc. & 1.114 & 1.126&1.136 & 1.144\\\hline
    eoc.(est) & 1.130 & 1.136 & 1.141 & 1.146\\ \hline
  \end{tabular}
  \caption{Estimated convergence order in the $L^2$ norm for the L1 scheme. Here $\beta = 0.8$ and $k = 4/3$, which is the grading suggested by Remark~\ref{rem:hatG_error}.}
  \label{tab:optimal_k}
\end{table}

The corresponding results  for CQ based on the backward Euler scheme are shown in Table~\ref{tab:L2_CQ}. Recall that the optimal convergence order for the L1 scheme and the CQ is different; see Remark~\ref{rem:cq_l1_conv}. It is better than linear $O(\tstep^{2-\beta})$ for the L1 scheme and linear for CQ. Further, linear convergence is obtained for the CQ scheme for $\beta > 0.5$ even with the uniform time-step, hence we only consider $\beta  = 0.2$ in the numerical results. Estimate \eqref{CQ_conv} predicts a convergence order $O(\tstep^{\beta+1/2}) = O(\tstep^{0.7})$. In Table~\ref{tab:L2_CQ}, the exact estimated order of convergence is slowly approaching $0.7$. However, our estimator converges at a slighly lower rate or $0.6$. The latter fits with Remark~\ref{rem:hatG_error} which predicts $O(\tstep^{\frac{1+\beta}2}) = O(\tstep^{0.6})$. Again, for $k  =2$, optimal convergence is obtained, i.e., in the case of CQ, linear convergence.

\begin{table}
  \centering
\subfloat[Uniform mesh with $\beta = 0.2$.]{
\begin{tabular}{|c|c|c|c|c|}\hline
 $N$ &  $E^1$ & eoc & $E^1_{\text{est}}$ & eoc\\\hline
$10$ & $0.0069$ &  & $0.11$ &  \\
$20$ & $0.0047$ & $0.55$ & $0.072$ & $0.58$ \\
$40$ & $0.0032$ & $0.57$ & $0.048$ & $0.58$ \\
$80$ & $0.0021$ & $0.59$ & $0.032$ & $0.58$ \\
$160$ & $0.0014$ & $0.60$ & $0.021$ & $0.59$   \\
$320$ & $0.00091$ & $0.61$ & $0.014$ & $0.59$   \\
$640$ & $0.00059$ & $0.62$ & $0.0094$ & $0.59$   \\
$1280$ & $0.00038$ & $0.63$ & $0.0063$ & $0.59$   \\
\hline
\end{tabular}
}
\subfloat[Graded mesh with $k = 2$ and $\beta = 0.2$.]{
\begin{tabular}{|c|c|c|c|c|}
\hline
 $N$ &  $E^1$ & eoc & $E^1_{\text{est}}$ & eoc\\\hline
$10$ & $0.0041$ &  & $0.031$ &  \\
$20$ & $0.0023$ & $0.86$ & $0.014$ & $1.1$ \\
$40$ & $0.0012$ & $0.9$ & $0.0067$ & $1.1$ \\
$80$ & $0.00064$ & $0.93$ & $0.0032$ & $1.09$ \\
$160$ & $0.00033$ & $0.95$ & $0.0015$ & $1.08$ \\  
\hline
\end{tabular}
}
  \caption{The numerical results for the backward Euler based CQ including the error measured in the $L^2$ norm at time $t = 1/2$ and the estimated order of convergence  with uniform and graded meshes. The fractional power is $\beta = 0.2$ throughout.}
  \label{tab:L2_CQ}
\end{table}

The corresponding results in  the $L^\infty$ error measure \eqref{eq:E2} evaluated at $t = 1/2$ are shown in Table~\ref{tab:Linf_L1} and Table~\ref{tab:Linf_CQ}. Again the estimator converges at the correct order. As indicated before,  convolution quadrature does not require a graded mesh to reach its optimal convergence order, i.e., linear,  in this norm at a fixed time away from the singularity.

\begin{table}
  \centering
\subfloat[Uniform mesh with $\beta = 0.8$.]{
\begin{tabular}{|c|c|c|c|c|}
\hline
 $N$ &  $E^2$ & eoc & $E^2_{\text{est}}$ & eoc\\\hline
10 & $0.012$ &  & $0.039$ & \\
20 & $0.0060$ & 0.95 &$0.020$ & 0.96\\
40 & $0.0030$ & 1.0 &$0.010$ & 0.97\\
80 & $0.0015$ & 1.0 &$0.0053$ & 0.98\\
160 & $0.00076$ &1.0  &$0.0027$ & 0.98\\
320 & $0.00038$ &1.0  &$0.0014$ & 0.98\\
%640 & $1.89\times 10^{-4}$ &1.0  &$6.88\times 10^{-4}$ & 0.98\\
%1280 & $9.44\times 10^{-5}$ &1.0  &$3.47\times 10^{-4}$ & 0.99\\
\hline
\end{tabular}
}
\subfloat[Graded mesh with $k = 2$ and $\beta = 0.8$.]{
\begin{tabular}{|c|c|c|c|c|}
\hline
 $N$ &  $E^2$ & eoc & $E^2_{\text{est}}$ & eoc\\\hline
10 & $0.011$ &  & $0.037$ & \\
20 & $0.0051$ & 1.14 &$0.017$ & 1.12\\
40 & $0.0023$ & 1.17 &$0.0076$ & 1.16\\
80 & $0.0010$ & 1.18 &$0.0034$ & 1.18\\
160 & $0.00044$ &1.19  &$0.0015$ & 1.19\\
320 & $0.00019$ &1.19  &$0.00065$ & 1.19\\
%640 & $8.37\times 10^{-5}$ &1.20  &$2.83\times 10^{-4}$ & 1.20\\
%1280 & $3.65\times 10^{-5}$ &1.20  &$1.23\times 10^{-4}$ & 1.20\\
\hline
\end{tabular}
}
  \caption{The numerical results for the L1 scheme including the error measured in the $L^\infty$ norm at time $t = 1/2$ and the estimated order of convergence for differerent choices of $\beta$ and with uniform and graded meshes.}
  \label{tab:Linf_L1}
\end{table}

\begin{table}
  \centering
\subfloat[Uniform mesh with $\beta = 0.8$.]{
\begin{tabular}{|c|c|c|c|c|}\hline
 $N$ &  $E^2$ & eoc & $E^2_{\text{est}}$ & eoc\\\hline
$10$ & $0.0088$ &  & $0.027$ &  \\
$20$ & $0.0045$ & $0.98$ & $0.014$ & $0.98$ \\
$40$ & $0.0022$ & $0.99$ & $0.007$ & $0.99$ \\
$80$ & $0.0011$ & $1$ & $0.0035$ & $0.99$ \\
$160$ & $0.00056$ & $1$ & $0.0018$ & $0.99$\\\hline
\end{tabular}   
}
\subfloat[Uniform mesh with $\beta = 0.2$.]{
\begin{tabular}{|c|c|c|c|c|}\hline
 $N$ &  $E^2$ & eoc & $E^2_{\text{est}}$ & eoc\\\hline
$10$ & $0.0026$ &  & $0.016$ &  \\
$20$ & $0.0013$ & $1.0$ & $0.0077$ & $1.0$ \\
$40$ & $0.00064$ & $1.0$ & $0.0038$ & $1.0$ \\
$80$ & $0.00032$ & $1.0$ & $0.0019$ & $1.0$ \\
$160$ & $0.00016$ & $1.0$ & $0.00095$ & $1.0$\\\hline
\end{tabular}
}
  \caption{The numerical results for the backward Euler based CQ including the error measured in the $L^\infty$ norm at time $t = 1/2$ and the estimated order of convergence  with uniform meshe for $\beta = 0.8$ and $\beta = 0.2$.}
  \label{tab:Linf_CQ}
\end{table}

\subsection{Fractional differential equation without an analytic solution}
{
For completeness, we investigate a fractional differential equation where an analytic expression for the solution does not exist:
\begin{equation}
  \label{eq:ode_01}  
\dt^\beta u(t)+u(t) = 2\cos(t), \qquad u(0) = 1, \qquad t \in (0,T).
\end{equation}
We repeat just a single experiment with the L1 scheme and $\beta =0.8$. Only the error $E^2(t)$ for $t = 1/2$ is investigated. The results for the uniform and graded meshes are shown Table~\ref{tab:Linf_L1_noexact} with the results analogous to the corresponding experiment in the previous section (see Table~\ref{tab:Linf_L1}). As the exact solution is not available, we make use of a graded, fine mesh  with  $k = 2$ to obtain that $u(1/2) = 1.3877\dots$, correct to the digits shown.}

\begin{table}
  \centering
\subfloat[Uniform mesh with $\beta = 0.8$.]{
\begin{tabular}{|c|c|c|c|c|}
\hline
 $N$ &  $E^2$ & eoc & $E^2_{\text{est}}$ & eoc\\\hline
10 & $0.015$ &  & $0.047$ & \\
20 & $0.0072$ & 1.0 &$0.024$ & 0.99\\
40 & $0.0035$ & 1.0 &$0.012$ & 1.0\\
80 & $0.0017$ & 1.0 &$0.0060$ & 1.0\\
160 & $0.00084$ &1.0  &$0.0030$ & 1.0\\
\hline
\end{tabular}
}
\subfloat[Graded mesh with $k = 2$ and $\beta = 0.8$.]{
\begin{tabular}{|c|c|c|c|c|}
\hline
 $N$ &  $E^2$ & eoc & $E^2_{\text{est}}$ & eoc\\\hline
10 & $0.016$ &  & $0.050$ & \\
20 & $0.0071$ & 1.14 &$0.023$ & 1.13\\
40 & $0.0032$ & 1.17 &$0.010$ & 1.16\\
80 & $0.0014$ & 1.19 &$0.0045$ & 1.18\\
160 & $0.00060$ &1.21  &$0.0020$ & 1.19\\
%320 & $0.000032$ &1.19  &$0.0001$ & 1.19\\
%640 & $8.37\times 10^{-5}$ &1.20  &$2.83\times 10^{-4}$ & 1.20\\
%1280 & $3.65\times 10^{-5}$ &1.20  &$1.23\times 10^{-4}$ & 1.20\\
\hline
\end{tabular}
}
  \caption{The numerical results for the L1 scheme applied to the problem \eqref{eq:ode_01} for which we have no exact solution. We include the error measured in the $E^2$ norm \eqref{eq:E2} at time $t = 1/2$ and the estimated order of convergence for  $\beta = 0.8$ with uniform and graded meshes.}
  \label{tab:Linf_L1_noexact}
\end{table}

\subsection{Fractional differential equation with non-smooth $f$}\label{sec:nonsmooth}
Including a simple correction term can recover optimal convergence if $f$ is smooth for both the L1 scheme \cite{Yan2018} and convolution quadrature \cite{Jin2017}. However, if $f$ is not globally smooth, non-uniform time-steps are needed for both the schemes. 

To illustrate this we consider the fractional differential equation
\begin{equation}
  \label{eq:ode_1}  
\dt^\beta u+\lambda u = f, \qquad u(0) = u^0, \qquad t \in (0,T),
\end{equation}
with $\lambda > 0$ a fixed constant and $f$ chosen so that the exact solution is given by
\[
u(t) = 1+t^\beta+H(t-r)t^\beta,
\]
where $H(\cdot)$ is the Heaviside function and $r \in (0,T)$ some constant. In this case the right-hand side $f$ is piecewise smooth
\[
f(t) = \lambda u(t) +\Gamma(\beta+1)(1+H(t-r)).
\]

Instead of a priori defining a graded mesh towards $t = 0$ and $t = r$, we will adaptively construct a non-uniform mesh using the developed estimators. We denote again by $E^2_{\text{est}}$ the $L^\infty$ estimate and given a parameter $\theta \in (0,1)$ we mark the interval $I_n$ if 
\[
E^2_{\text{est}}(t_{n+1}) \geq \theta \max_{j = 1,\dots,N} E^2_{\text{est}}(t_{j}).
\]
To obtain a refined mesh, each marked interval is split into two. 

To investigate convergence, we use as the error measure the maximum error
\[
e^{\max} = \max_{j = 1,\dots,N} E^2(t_j)
\]
and the corresponding a posteriori error estimate
\[
e^{\max}_{\text{est}} = \max_{j = 1,\dots,N} E^2_{\text{est}}(t_j).
\]

\begin{figure}
  \centering
  \includegraphics[width=.46\textwidth]{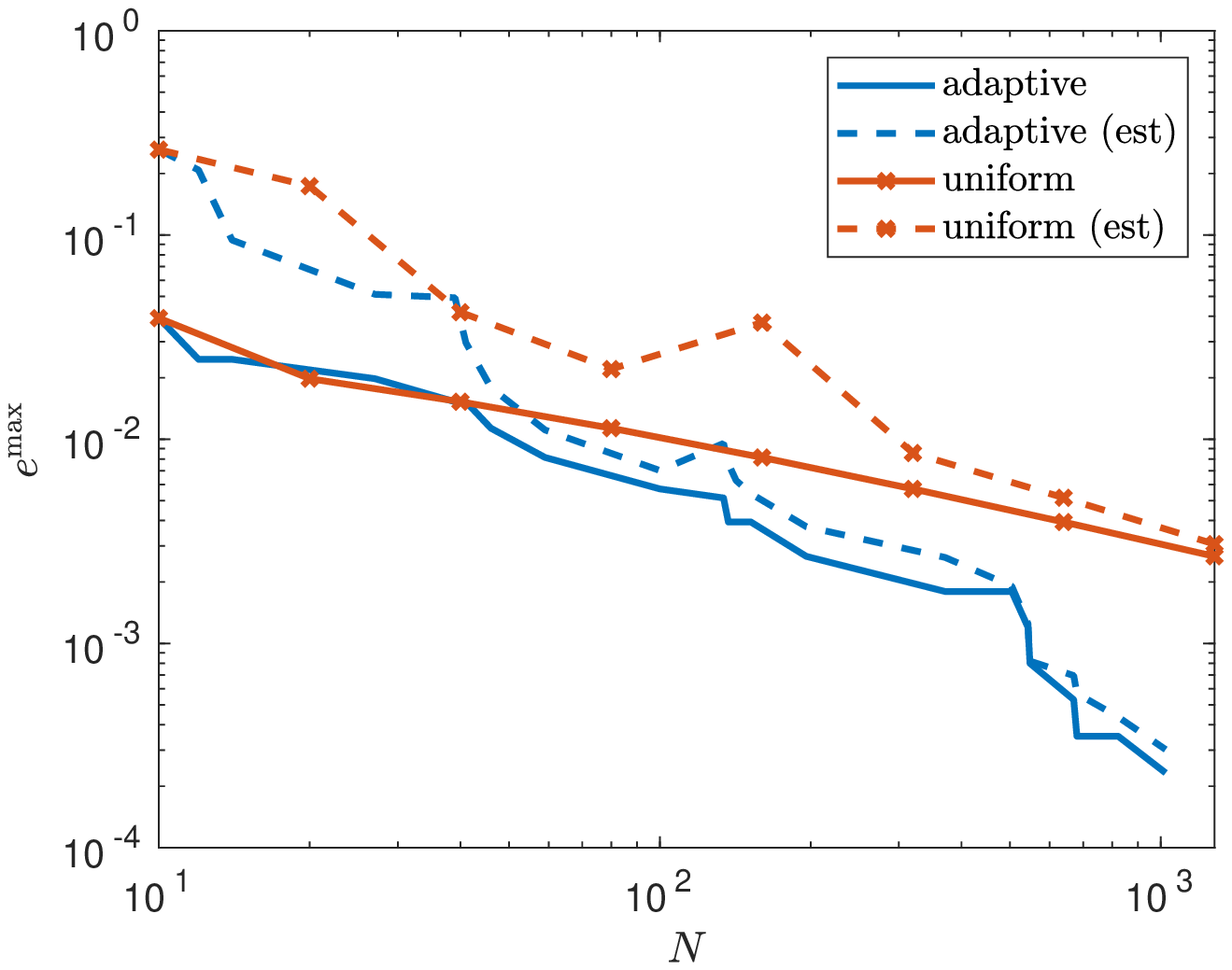}
  \includegraphics[width=0.46\textwidth]{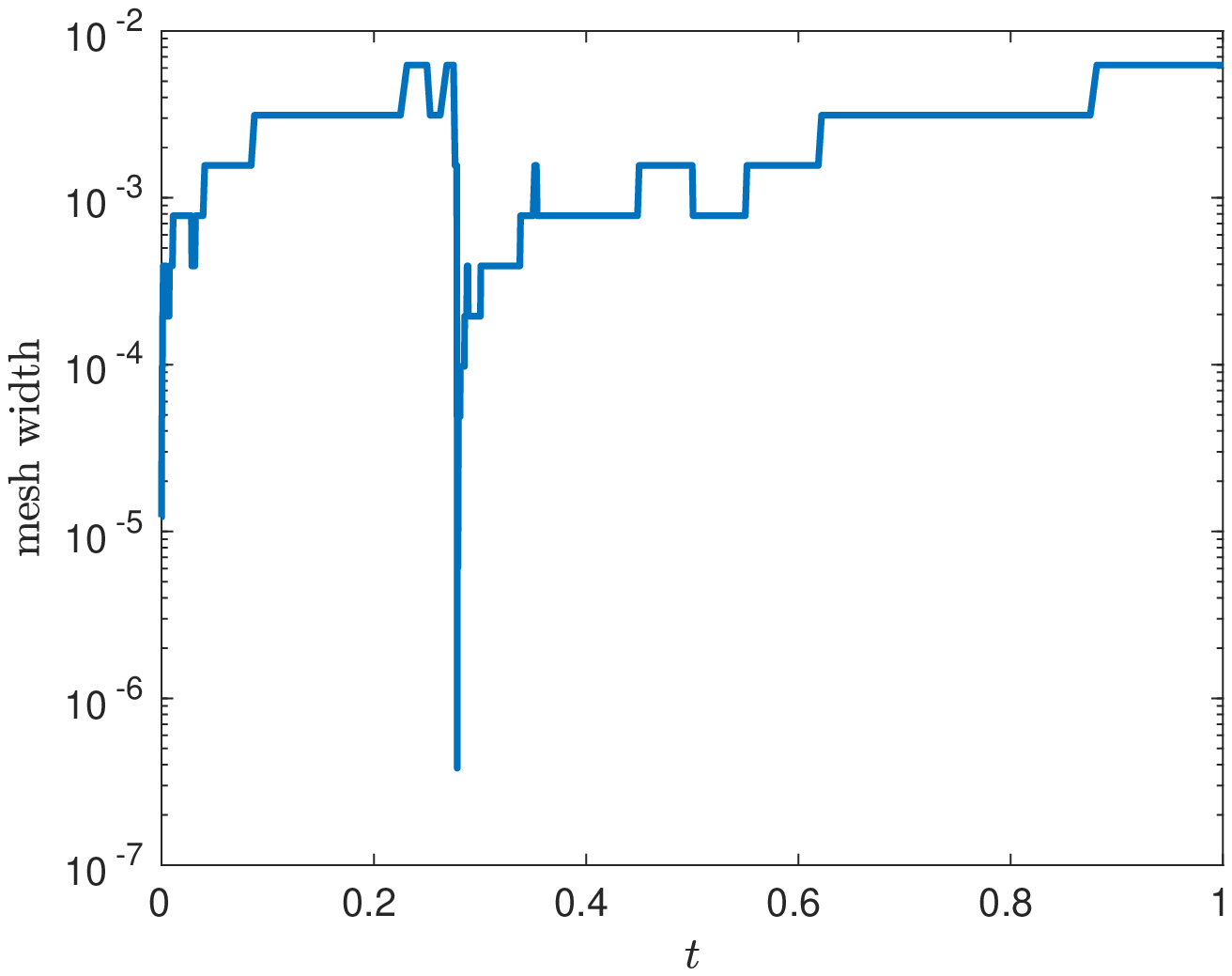}
  \caption{On the left we compare the convergence of the L1 adaptive scheme  with the uniform discretization. The fractional power is $\beta = 0.6$. The time-steps chosen adaptively are shown on the right. Refinement near the singularities at $t = 0$ and $t = 0.28$ can clearly be seen.}
  \label{fig:L1_max}
\end{figure}

\begin{figure}
  \centering
  \includegraphics[width=.6\textwidth]{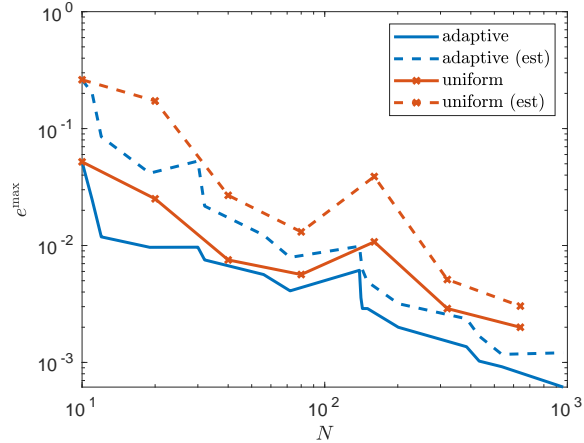}
  \caption{Comparison of the convergence of the CQ adaptive scheme compared with the uniform discretization. The fractional power is $\beta = 0.6$.}
  \label{fig:CQ_max}
\end{figure}

We perform experiments with $\beta = 0.6$, $\theta  = 0.75$,  and the singularity at $r = 0.28$. 
The comparison of the convergence in the above norms of the uniform schemes with the above described adaptive scheme are shown on the left of Figure~\ref{fig:L1_max} for the L1 scheme and in Figure~\ref{fig:CQ_max} for CQ. We see that the adaptive scheme in both cases does considerably better than the scheme with a uniform time step. The difference would be much more pronounced for a smaller $\beta$ and hence less smooth $u$.
On the right of  Figure~\ref{fig:L1_max} we plot the time-steps chosen by the scheme. Clearly the adaptive scheme was able to locate the singularities near $t =0$ and $t = r$. 

\subsection{Subdiffusion equation with FEM in space}
In this section we consider a subdiffusion equation on a 1D domain $\Omega = (0,\pi)$: Find $u(t) \in H^1_0(\Omega)$ such that
\begin{equation}
  \label{eq:pde_0}  
  \begin{aligned}    
\dt^\beta u-\Delta  u &= f & &\text{for } (t,x) \in (0,T) \times \Omega,\\
u(t,x) &= 0 & &\text{for } (t,x) \in (0,T) \times \partial\Omega,\\
 u(0,x) &= u^0(x) & &\text{for } x \in \Omega.
  \end{aligned}
\end{equation}
The right-hand side  $f$ is chosen so that the exact solution is given by
\[
u(t,x) = (1+t^\beta+H(t-r)t^\beta)\sin x
\]
with  $r \in (0,T)$ some constant. In this case the right-hand side $f$ is piecewise smooth in time
\[
f(t,x) =  u(t,x) +\Gamma(\beta+1)(1+H(t-r))\sin x.
\]

Choosing a finite dimensional subspace $V_h \subset H^1_0(\Omega)$, the Galerkin discretization of \eqref{eq:pde_0} results in the problem: Find $\ub(t) \in V_h$
\begin{equation}
  \label{eq:pde_0_gal}
  \begin{aligned}    
  \dt^\beta \ub + A_h \ub &= P_h f & \text{for } t \in (0,T)\\
\ub(0) &= P_hu^0. & 
  \end{aligned}
\end{equation}
$A_h: V_h \rightarrow V_h$ is the Galerkin discretization of the Dirichlet Laplacian defined by
\[
(A_h u, v)_{L^2(\Omega)} = (\nabla u, \nabla v)_{L^2(\Omega)} \qquad
\text{for all } u,v \in V_h.
\]
The operator $P_h : L^2(\Omega) \rightarrow V_h$ is the $L^2$-projection defined by
\[
(P_h u, v)_{L^2(\Omega)} = (u,v)_{L^2(\Omega)} \qquad \text{for all } v \in V_h.
\]
The operator $A_h$ satisfies all the assumptions required so that this problem fits into the setting \eqref{eq:sub_diff} with $A = A_h$, $V = H = V_h$.
 The above setting is used for convenience and in order to demonstrate the behaviour of the time-estimators derived in this paper. 
A complete a posteriori and adaptive treatment of \eqref{eq:pde_0}  will require dynamical change of the finite element spaces (i.e. variable with $n$ discrete spaces $V_h^n$)  in the scheme definition and an analysis taking into account the spatial discretisation error. 
Given the ``multistep nature'' of the time-discretisation of  the fractional equation \eqref{eq:pde_0}, even the definition of the scheme requires special attention. 
Still, point-wise  representations of the fully-discrete methods are possible using the elliptic reconstruction, \cite{MN:siam2003}. In this case the application of the framework derived herein  will be applicable. 
However, this is not a straightforward task, compare with \cite {MR2404052, LM:mc2006} for parabolic problems, and it will be the subject of a forthcoming work. 

In the numerical experiments we fix $V_h$ to a piecewise quadratic finite element space on a uniform spatial mesh with mesh-width $h > 0$. The discretization in time is achieved using the adaptive L1 scheme described in the previous section up to time $T = 1$. As the error measure we use
\begin{equation}
  \label{eq:max_fem_norm}
e^{\max} = \max_{j = 1,\dots,N} \|P_hu(t_j)-U_j\|_{L^2(\Omega)},  
\end{equation}
whereas the a posteriori measure is
\begin{equation}
  \label{eq:max_fem_adap_norm}  
e^{\max}_{\text{est}} = \max_{j = 1,\dots,N} E^2_{\text{est}}(t_j),
\end{equation}
where $E^2_{\text{est}}(t_j)$ is obtained using Theorem~\ref{thm:apost_linf}.

The numerical results for $\beta = 0.2$ and $r = 0.28$ are shown in Figure~\ref{fig:fem_plot}. In the mark-and-refine scheme we again set $\theta = 0.75$. Similar results as in the previous section are obtained. 
\begin{figure}
  \centering
  \includegraphics[width=0.4\textwidth]{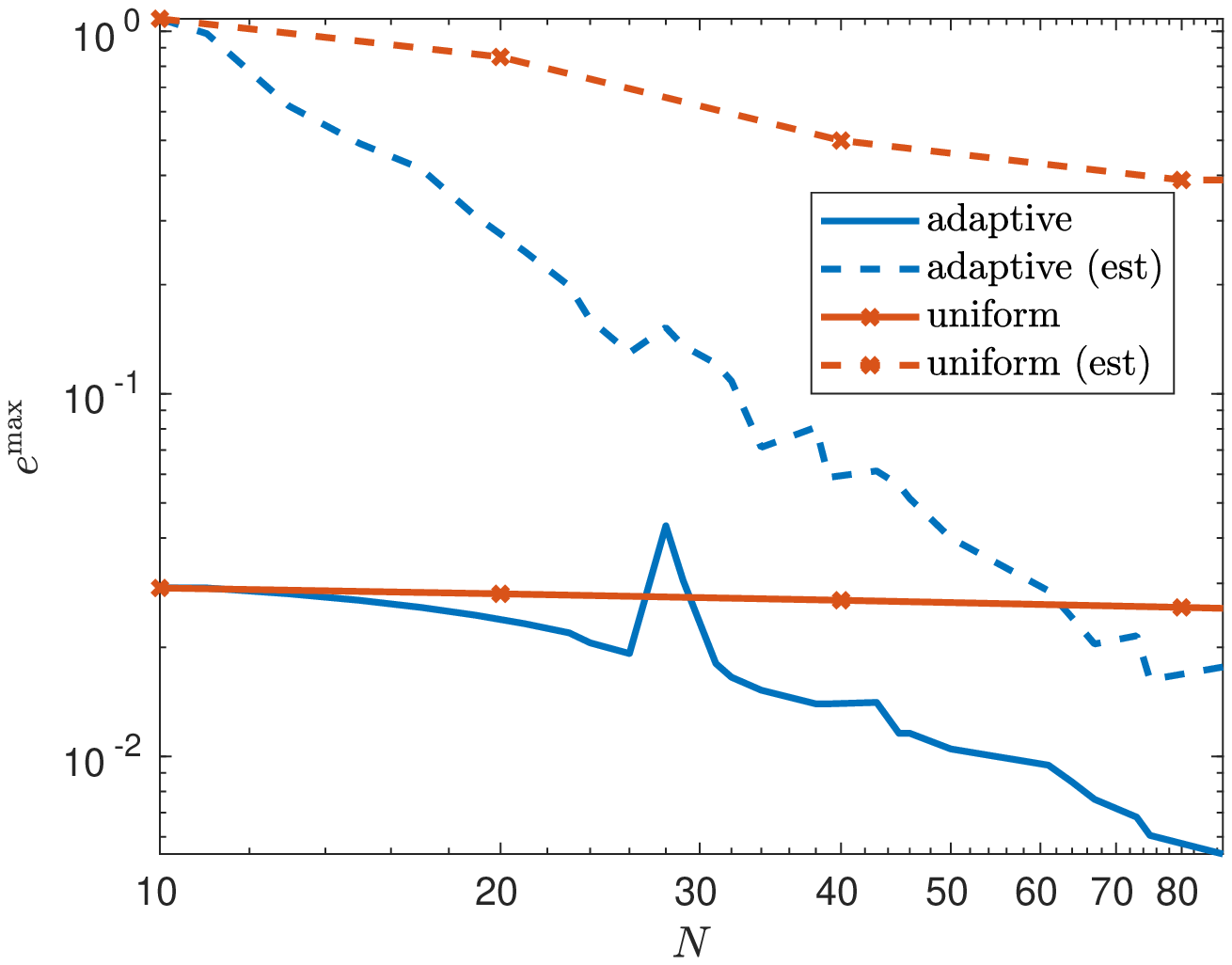}
  \includegraphics[width=0.55\textwidth]{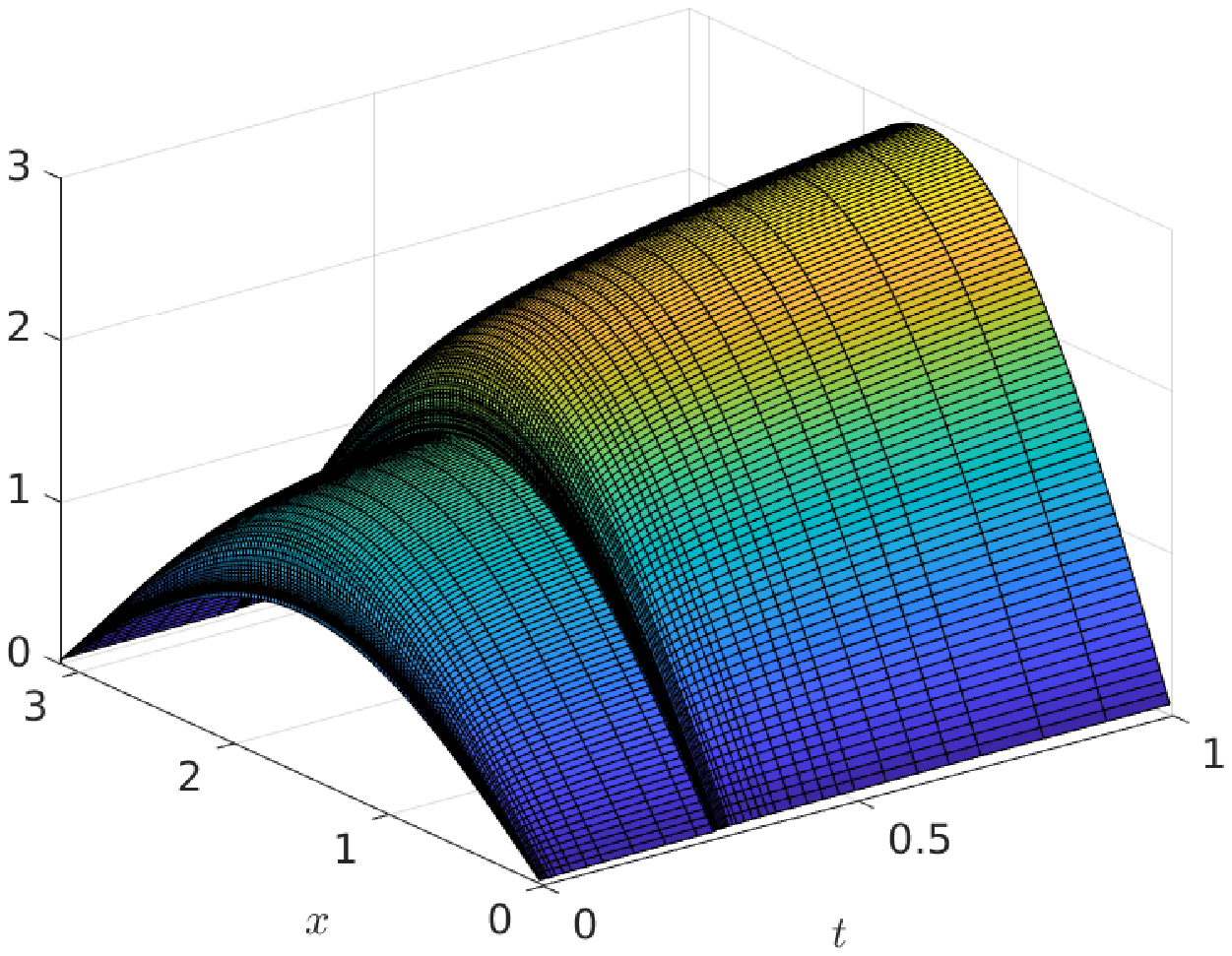}
  \caption{Numerical results for the subdiffusion equation \eqref{eq:pde_0} with $\beta = 0.2$. On the left we plot the convergence of the adaptive scheme in the maximum norm \eqref{eq:max_fem_norm} and \eqref{eq:max_fem_adap_norm} and compare it with the scheme with uniform time-step.  On the right is a plot of the solution where in the visible grid adaptive refinement towards the singularities at $t = 0$ and $t = 0.28$ can be discerned.}
  \label{fig:fem_plot}
\end{figure}
\appendix

\section{Non-uniform convolution quadrature}
Let $K$ be  a sectorial operator, i.e., for some $\theta \in (0,\pi/2)$ and $\mu \in \mathbb{R}$, K is analytic in $\mathbb{C} \setminus \{ z \;;\; |\arg z| < \pi -\theta\}$ and satisfies
\[
|K(s)| \leq C |s|^{\mu} \qquad s \in \mathbb{C} \setminus \{ z \;;\; |\arg z| < \pi -\theta\}.
\]
Denoting by $k = \mathscr{L}^{-1}K$ the inverse Laplace transform of $K$, we consider the convolution
\[
u(t) = \int_0^tk(t-\tau)g(\tau)d\tau,
\]
If $\mu < 0$, $k$ is integrable otherwise the integral needs to be understood as a Hadamard finite part integral or equivalently setting $K_m(s) = s^{-m} K(s)$ with $m = \lceil \mu \rceil$ we have
\[
u(t) = \frac{d^m}{dt^m} \int_0^tk_m(t-\tau)g(\tau)d\tau,
\]
where $k_m = \mathscr{L}^{-1}K_m$ .

Replacing $k$ with the inverse Laplace transform of $K$ and exchanging integrals gives
\begin{equation}\label{eq:var_const}
u(t) = \frac1{2\pi i} \int_{\sigma\pm i\infty} K(s) \int_0^t e^{s(t-\tau)} g(t) ds
=  \frac1{2\pi i} \int_{\sigma\pm i\infty} K(s) y(t;s)ds,
\end{equation}
where  $\sigma > 0$ is a constant and $y$ solves the ODE
\[
y' = sy+g \qquad y(0) = 0.
\]
This calculation can be justified for $\mu < 0$. For $\mu \geq 0$, it is a formal argument that nevertheless leads to a well-defined numerical scheme.

Convolution quadrature is obtained by discretizing the ODE and substituting the result in \eqref{eq:var_const}. Applying backward Euler discretization to the ODE gives
\[
\frac1{\tstep_{n-1}}(y_n - y_{n-1}) = s y_n +g(t_n) \qquad y_0 = 0. 
\]
Solving the recursion gives
\[
  \begin{split}    
y_{n+1} &= (1-s\tstep_n)^{-1}\left[y_n+\tstep_{n-1} g(t_{n+1})\right]\\
    &=  \sum_{j = 0}^n \tstep_j\prod_{k = j}^n(1-s\tstep_{k})^{-1}g(t_{j+1}),
  \end{split}
\]
Hence 
\begin{equation}
  \label{eq:gen_cq}  
u_{n+1} = \sum_{j = 0}^n \cqw_{n,j}(K) g(t_{j+1})
\end{equation}
where
\[
  \begin{split}    
\cqw_{n,j}(K) &= \tstep_j\frac1{2\pi i} \int_{\sigma- i\infty}^{\sigma+i\infty}  K(s) \prod_{k = j}^n(1-s\tstep_{k})^{-1}ds\\
&= \tstep_j\frac1{2\pi i} \int_{\mathcal{C}}K(s) \prod_{k = j}^n(1-s\tstep_{k})^{-1}ds,
  \end{split}
\]
where $\mathcal{C}$ is a negatively oriented contour contained in the right-hand complex plane surrounding the poles at $s = \tstep_k^{-1}$. In this form, we see that no condition on the growth parameter $\mu$ of $K$ is needed for the weights to be well-defined.

An important property of convolution quadrature is the composition rule. Let $K(s) = K_1(s)K_2(s)$, then
\[
u_{n+1} = \sum_{j = 0}^n \cqw_{n,j}(K) g(t_{j+1})
= \sum_{j = 0}^n \cqw_{n,j}(K_2) \sum_{\ell = 0}^j \cqw_{j,\ell}(K_1) g(t_{\ell+1}).
\]
This property was shown  in \cite{LopS16} for $K_1(s) = K_2^{-1}(s)$ by using properties of divided differences. We explain the steps for general $K_1$ and $K_2$. 
\[
  \begin{split}    
\cqw_{n,j}(K) &= \tstep_j \prod_{k = j}^n(-\tstep_k)^{-1}\frac1{2\pi i} \int_{\mathcal{C}}K(s) \frac1{\prod_{k = j}^n(\tstep_k^{-1}-s)}ds\\
&= -\left(\prod_{k = j+1}^n(-\tstep_k)^{-1}\right)\frac1{2\pi i} \int_{\mathcal{C}}K(s) \frac1{\prod_{k = j}^n(s-\tstep_k^{-1})}ds\\
&= \left(\prod_{k = j+1}^n(-\tstep_k)^{-1}\right)[\tstep_j^{-1},\dots,\tstep_n^{-1}]K,
\end{split}
\]
where $[x_0,\dots,x_n]f$ are Newton's divided differences. For the equivalence of the above contour integral and divided differences see \cite[Equation (51)]{deBoor}; note that the negative sign is absorbed due to the contour $\mathcal{C}$ being negatively oriented.

 To show the composition rule we need the Leibniz rule for divided differences \cite[Corollary 28]{deBoor}
\[
[x_0,\dots,x_n](fg) = \sum_{k = 0}^n [x_0,\dots,x_k]f [x_k,\dots,x_n]g.
\]
Hence
\[
  \begin{split}
    \cqw_{n,j}(K) &=\left(\prod_{k = j+1}^n(-\tstep_k)^{-1}\right)[\tstep_j^{-1},\dots,\tstep_n^{-1}]K_1K_2\\
&= \sum_{\ell = j+1}^n\left(\prod_{k = j+1}^n(-\tstep_k)^{-1}\right)[\tstep_j^{-1},\dots,\tstep_\ell^{-1}] K_1 [\tstep_{\ell}^{-1},\dots,\tstep_n^{-1}]K_2\\
&= \sum_{\ell = j+1}^n\left(\prod_{k = j+1}^\ell(-\tstep_k)^{-1}\right)[\tstep_j^{-1},\dots,\tstep_\ell^{-1}] \left(\prod_{k = \ell+1}^n(-\tstep_k)^{-1}\right)K_1 [\tstep_{\ell}^{-1},\dots,\tstep_n^{-1}]K_2\\
&= \sum_{\ell = j+1}^n \cqw_{\ell,j}(K_1) \cqw_{n,\ell}(K_2).
  \end{split}
\]
Substituting this into \eqref{eq:gen_cq} gives the composition rule.

Let us now apply all this to the fractional derivative where $K(s) = s^\beta = K_1(s) K_2(s)$ with $K_1(s) = s^{\beta-1}$, $K_2(s) = s$. Using the divided difference definition we see that
\[
\cqw_{n,n}(K_2) = \tstep_n^{-1}, \; \cqw_{n,n-1}(K_2) = -\tstep_n^{-1}, \cqw_{n,j} = 0, \qquad j< n-1\]
whereas $\cqw_{n,j}(K_1)$ are, by definition, as in \eqref{eq:cq_weights}.

The definition of convolution used in this section gives rise to the Riemann-Liouville derivative
\[
^\text{RL}_{\;\;\;0}\dt^\beta u = \frac{d}{dt} \frac1{\Gamma(1-\beta)} \int_0^t (t-\tau)^{-\beta} u(\tau) d\tau.
\]
In order to apply convolution quadrature to the Caputo derivative we rewrite it as 
a Riemann-Liouville derivative using the identity
\[
^\text{C}_{0}\dt^\beta u =  ^\text{RL}_{\;\;\;0}\dt^\beta (u-u(0))
\]
valid for sufficiently smooth $u$. Hence given a sequence of values $U_0, \dots, U_N$ we apply convolution quadrature to the sequence $U_j-U_0$, $j = 0,\dots,N$ giving
\[
  \begin{split}
 \sum_{j = 0}^n \cqw_{n,j}(K) (U_{j+1}-U_0) &= \sum_{j = 0}^n \cqw_{n,j}(K_1) \sum_{\ell = 0}^j \cqw_{j,\ell}(K_2) (U_{\ell+1}-U_0)\\
&= \sum_{j = 0}^n \cqw_{n,j}(K_1) \frac1{\tstep_j}\left(U_{j+1}-U_j\right),
  \end{split}
\]
using the composition rule. This is the  representation \eqref{eq:CQ_defn} used in this paper.

%\bibliographystyle{amsplain}
%\bibliographystyle{abbrv}
%\bibliography{subdiff_ref.bib}

\def\cprime{$'$}

\end{document}